\documentclass[a4paper, 12pt]{amsart}
\usepackage{amsmath,amssymb,amsthm}

\usepackage{url} 
\usepackage{amscd}
\usepackage{array,enumerate}
\newcommand{\ip}[1]{\mathopen{\langle}#1\mathclose{\rangle}}
\newtheorem{Thm}{Theorem}[section]
\newtheorem{Prop}[Thm]{Proposition}
\newtheorem{Lem}[Thm]{Lemma}

\newtheorem{Cor}[Thm]{Corollary}
\newtheorem{Thmint}{Theorem}[section]

\theoremstyle{definition}
\newtheorem{Rem}[Thm]{Remark}

\newtheorem{Def}[Thm]{Definition}

\newcommand{\0}{\{0\}}
\newcommand{\Cs}{C$^\ast$}

\newcommand{\id}{\mbox{\rm id}}

\newcommand{\rc}{\mathop{\rtimes _{\mathrm r}}}

\newcommand{\vnc}[1]{\mathop{{\bar \rtimes}_{#1}}}

\newcommand{\rca}[1]{\mathop{\rtimes _{{\mathrm r}, #1}}}

\newcommand{\Int}[2]{{\rm Int}(#1 \subset #2)}
\newcommand{\E}{{\mathbf E}}

\newcommand{\IB}{\mathbb B}
\newcommand{\IC}{\mathbb C}

\newcommand{\IK}{\mathbb K}
\newcommand{\IM}{\mathbb M}
\newcommand{\IN}{\mathbb N}
\newcommand{\IR}{\mathbb R}
\newcommand{\IT}{\mathbb T}

\newcommand{\IZ}{\mathbb Z}

\newcommand{\fx}{\mathfrak x}

\newcommand{\n}{\mathfrak n}
\newcommand{\fU}{\mathfrak U}

\newcommand{\fu}{\mathfrak u}
\newcommand{\fv}{\mathfrak v}
\newcommand{\fw}{\mathfrak w}

\newcommand{\fH}{\mathfrak H}

\newcommand{\cL}{\mathcal L}

\newcommand{\cV}{\mathcal V}
\newcommand{\cM}{\mathcal M}

\newcommand{\btimes}{\mathbin{\bar{\otimes}}}

\newcommand{\acts}{\curvearrowright}

\newcommand{\ad}{\mathrm{Ad}}
\DeclareMathOperator{\Aut}{\mathop{Aut}}

\DeclareMathOperator{\spa}{\mathop{span}}

\DeclareMathOperator{\PSL}{PSL}
\DeclareMathOperator{\SL}{SL}

\DeclareMathOperator{\Cso}{{\rm C}^\ast}

\title[]{Crossed product splitting of intermediate operator algebras via $2$-cocycles}

\author{Yuhei Suzuki}
\address{Department of Mathematics, Faculty of Science, Hokkaido University,
Kita 10, Nishi 8, Kita-Ku, Sapporo, Hokkaido, 060-0810, Japan}

\email{yuhei@math.sci.hokudai.ac.jp}
\begin{document}
\maketitle

\begin{abstract}
We investigate the \Cs-algebra inclusions $B \subset A \rc \Gamma$
arising from inclusions $B \subset A$ of $\Gamma$-\Cs-algebras.
 The main result shows that, when $B \subset A$ is \Cs-irreducible in the sense of R{\o}rdam,
and is centrally $\Gamma$-free in the sense of the author, then after tensoring with the Cuntz algebra $\mathcal{O}_2$,
all intermediate \Cs-algebras $B \subset C\subset A \rc \Gamma$ enjoy
a natural crossed product splitting
\[\mathcal{O}_2\otimes C=(\mathcal{O}_2 \otimes D) \rca{\gamma, \fw} \Lambda\]
for $D:= C \cap A$, some $\Lambda<\Gamma$, and a subsystem $(\gamma, \fw)$ of a unitary perturbed cocycle action $\Lambda \acts \mathcal{O}_2\otimes A$. 
As an application, we give a new Galois's type theorem for the Bisch--Haagerup type inclusions
\[A^K \subset A\rc \Gamma\]
for actions of compact-by-discrete groups $K \rtimes \Gamma$ on simple \Cs-algebras.

Due to a K-theoretical obstruction, the operation $\mathcal{O}_2\otimes -$ is necessary to obtain the clean splitting.
Also, in general $2$-cocycles $\fw$ appearing in the splitting cannot be removed even further tensoring with any unital (cocycle) action.
We show them by examples, which further show that $\mathcal{O}_2$ is a minimal possible choice.

We also establish a von Neumann algebra analogue, where $\mathcal{O}_2$ is replaced by the type I factor $\IB(\ell^2(\IN))$.
\end{abstract}

\tableofcontents

\section{Introduction}
Inclusions of mathematical objects play significant roles in many contexts.
For instance, Galois theory deduces the solvability problem of a polynomial to a 
problem on the symmetry of the associated inclusion of fields.
In finite group theory, the study of certain subgroups play crucial roles:
Sylow's theorem is definitely a basic tool in the study of finite groups,
and the study of subgroups is also crucial in classification theory of finite simple groups.

Operator algebra theory is no exception to this general philosophy.
Studies of subalgebras (e.g., maximal amenable subalgebras \cite{Pop83}) and embeddings (e.g., injective enveloping \Cs-algebras \cite{Ham85}, constructions and classification of embeddings)
are of great interest and deeply studied by a wide range of operator algebraists,
and some of them become key players of major breakthroughs.
Highlights include Jones's subfactor theory \cite{Jon},
Popa's deformation/rigidity theory \cite{PopICM},
accessible characterizations on \Cs-simplicity by Kalantar--Kennedy \cite{KK} and their striking consequences \cite{BKKO},
the non-commutative Nevos--Zimmer theorem of Boutonnet--Houdayer \cite{BH23}, \cite{HouICM},
and an abstract approach to the classification theorem of unital simple separable nuclear \Cs-algebras \cite{CGS}.

An important invariant for an inclusion $B \subset A$ of operator algebras
is the \emph{lattice of intermediate operator algebras} $\Int{B}{A}$.
However this invariant is usually extremely difficult to calculate.
A reason of the difficulty comes from the construction of operator algebras.
When one constructs an operator algebra from another mathematical object (e.g., groups, dynamical systems),
the construction involves the \emph{completion}.
(This is also a reason why finite dimensional approximation properties are important in operator algebra theory, see the book \cite{BO}.)
As compensation for being able to use the operator algebra framework,
the resulting algebra contains transcendental elements.
As a result, the lattice becomes extremely hard to access.

However in certain situations, one can show that
there are only \emph{canonical} (``algebraic'') intermediate operator algebras via highly non-trivial arguments.
The purpose of this article is to give a new result in this direction.
Before stating our main results, let us briefly review backgrounds on classification results of intermediate operator algebras.

Basically we have three previously known general theorems giving a complete description of intermediate operator algebras.
For simplicity of the statements, here we state them in the von Neumann algebra context.

\noindent
(1) {\bf Tensor Splitting Theorem (Ge--Kadison \cite{GK}):} For a factor $M$ and any von Neumann algebra inclusion $B \subset A$, 
the map
\[M\btimes - \colon {\rm Int}(B\subset A) \rightarrow {\rm Int}(M \btimes B \subset M \btimes A); \quad C \mapsto M \btimes C\] is a lattice isomorphism. The C$^\ast$-algebra version is proved by Zacharias \cite{Zac} and Zsido \cite{Zsi} (see also \cite{SuzMAAN}, Theorem A.3 for the non-unital case).\\
(2-1) {\bf Galois Correspondence Theorem for compact groups (Izumi--Longo--Popa \cite{ILP}):} Consider any faithful compact group action $K \curvearrowright M$ on a factor 
with $(M^K)' \cap M=\mathbb{C}$. Here $M^K$ denotes the fixed point algebra.
Then the map
\[{\rm Sub}(K) \rightarrow {\rm Int}(M^K\subset M);\quad L\mapsto M^L\] gives an anti-lattice isomorphism.
The C$^\ast$-algebra version is recently proved by Mukohara \cite{Mu} for isometrically shift-absorbing actions (see \cite{GS2}).\\
(2-2) {\bf Galois Correspondence Theorem for discrete groups (Izumi--Longo--Popa \cite{ILP}):}
For any pointwise outer action $\alpha$
of a discrete group $\Gamma$ on a factor $M$,
the map
\[{\rm Sub}(\Gamma) \rightarrow \Int{M}{M\vnc{\alpha}\Gamma};\quad \Lambda \mapsto M\vnc{\alpha} \Lambda\]
is a lattice isomorphism.
The C$^\ast$-algebra version is proved by Cameron--Smith \cite{CS}.\\
(3) {\bf Crossed Product Splitting Theorem (Suzuki \cite{SuzCMP}):}
For a centrally $\Gamma$-free inclusion of $\Gamma$-von Neumann algebras $N\subset (M, \alpha)$, the map
\[\Int{N}{M}^\Gamma \rightarrow \Int{N\vnc{\alpha}\Gamma}{M\vnc{\alpha}\Gamma};\quad C\mapsto C \vnc{\alpha}\Gamma\]
gives a lattice isomorphism. Here $\Int{N}{M}^\Gamma$ denotes the fixed point set.
(See also \cite{Pac} for classical results for trace preserving actions on finite von Neumann algebras, and \cite{AGG}, \cite{CS}, \cite{Och} etc.~ for some partial generalizations.)
The \Cs-algebra version is also shown in \cite{SuzCMP}.

Recently Bader--Boutonnet--Houdayer \cite{BBH} proved a version of the Margulis factor theorem
on the crossed product algebra level, in which the Crossed Product Splitting Theorem played a crucial role.
It follows that the inclusions
\[L(\PSL(n, \IZ)) \subset L^\infty(\PSL(n, \IR)/P_n)\bar{\rtimes} \PSL(n, \IZ);\quad n=3, 4, \ldots\]
are mutually non-isomorphic.
Thus the longstanding classification problem of the factors $L(\PSL(n, \IZ))$; $n=3, 4, \ldots$
could be deduced to the reconstruction of the above inclusions
from their intrinsic structures.

In this article, we propose a new complete description result of intermediate operator algebra lattices,
which unifies Theorems (2-2) and (3).

For the terminologies and notations used below, see Section \ref{section:Prelim}.
\begin{Thmint}\label{Thmint:Cs}
Let $\Gamma$ be a discrete group with Haagerup--Kraus's approximation property $($AP$)$.
Let $B\subset (A, \alpha)$ be a \Cs-irreducible, centrally $\Gamma$-free inclusion of unital $\Gamma$-\Cs-algebras.
Then for any $C\in \Int{B}{A\rca{\alpha} \Gamma}$,
there is a subgroup $\Lambda<\Gamma$ and an $(\mathcal{O}_2\otimes C, \Lambda)$-respecting map $\fu \colon \Lambda \rightarrow (\mathcal{O}_2\otimes A)^{\rm u}$ with
\[\mathcal{O}_2\otimes C =(\mathcal{O}_2\otimes (C \cap A)) \rca{{}^\fu \alpha, {\rm d}\fu} \Lambda.\]
\end{Thmint}
The theorem states that, under suitable assumptions on a $\Gamma$-\Cs-algebra inclusion $B\subset (A, \alpha)$,
after tensoring with the Cuntz algebra $\mathcal{O}_2$ (equipped with the trivial action),
\emph{all} intermediate \Cs-algebras of the inclusion $B \subset A \rca{\alpha}\Gamma$ split into a twisted crossed product in a canonical way.
As many deep structural results are available for the (twisted) crossed products (see the books \cite{Pedbook}, \cite{BO}, \cite{Wil}, \cite{AP} for instance)
and the associated (unique) cocycle actions are given on the purely algebraic level,
the result is helpful to understand the inclusions of the above form.
The proof is given in Section \ref{section:proof}.
We extend Theorem \ref{Thmint:Cs} to twisted and/or non-unital case in Section \ref{section:general} (see Theorem \ref{Thm:generalCs}).

We note that both the $2$-cocycle $\fw$ and $\mathcal{O}_2$ appearing in the splitting
cannot be removed. We prove them by examples in Sections \ref{section:nonexact} and \ref{section:O2} respectively.
We believe that the two results below are of independent interest.
The next result is an evidence that the $2$-cocycles cannot be removed in our theorem,
even further tensoring with any unital cocycle actions.
This is an application of the recent developments on amenable actions on simple \Cs-algebras \cite{Suzeq}, \cite{OS}.
\begin{Thmint}
Let $\Gamma$ be a non-exact countable group. Then there is a cocycle action
$(\alpha, \fw) \colon \Gamma \acts \mathcal{O}_2$ with the following property.
For any cocycle action $(\beta, \fx) \colon \Gamma \acts A$ on a unital \Cs-algebra,
the diagonal cocycle action $(\alpha\otimes \beta, \fw\otimes \fx)$ is not exterior equivalent to a genuine action.
\end{Thmint}
The next result shows that a new tensor component is necessary to obtain the clean splitting,
and $\mathcal{O}_2$ is a (virtually) minimal possible choice.
\begin{Thmint}[For the precise statements and the proofs, see Section \ref{section:O2}]
There is an inclusion $B \subset (A, \alpha)$ as in Theorem \ref{Thmint:Cs} with the following property.
There is $C\in \Int{B}{A\rca{\alpha}\Gamma}$ such that for any unital \Cs-algebra $D$ with $[1_D]_0\neq 0$,
the tensor product $C\otimes D$ does not enjoy a natural twisted crossed product splitting. 
\end{Thmint}
Still a partial splitting result is available without adding a new tensor component. See Remark \ref{Rem:partial}.

By applying Theorem \ref{Thmint:Cs} to the Bisch--Haagerup type inclusions $A^K \subset A \rca{\alpha} \Gamma$ \cite{BH}
arising from compact-by-discrete group actions $\alpha \colon K \rtimes_\sigma \Gamma \acts A$, we establish the following new Galois's type theorem.

\begin{Thmint}\label{Thmint:Galois}
Let $\alpha \colon K \rtimes_\sigma \Gamma \acts A$ be an action on a simple separable \Cs-algebra
such that $\id_{\mathcal{O}_2} \otimes \alpha$ is isometrically shift-absorbing.
Then there is
a canonical lattice isomorphism
\[\cL(K, \Gamma, \sigma) \cong \Int{A^K}{A\rca{\alpha} \Gamma},\]
where $\cL(K, \Gamma, \sigma)$ is a lattice associated to $\sigma \colon \Gamma \rightarrow {\rm Aut}(K)$.
\end{Thmint}
The definition of $\cL(K, \Gamma, \sigma)$ and an explicit isomorphism
are given in Section \ref{section:Galois}.

We also establish the von Neumann algebra version of Theorem \ref{Thmint:Cs}, where $\mathcal{O}_2$ is replaced by
a type I$_\infty$ factor (Corollary \ref{Cor:vnsp}).
Because the K$_0$-obstruction is very small in factors, under assumptions on the Murray--von Neumann type,
we do not need to add a tensor product component.
\begin{Thmint}\label{Thmint:vn}
Let $N \subset (M, \alpha)$ be an irreducible, centrally $\Gamma$-free subfactor.
\begin{enumerate}[\upshape(1)]
\item If $M$ is of finite type, then
for any $C\in \Int{N}{M\vnc{\alpha}\Gamma}$,
there is a subgroup $\Lambda<\Gamma$ and a $((C\cap M), \Lambda)$-respecting map $\fu \colon \Lambda \rightarrow M^{\rm u}$ with
\[C =(C \cap M) \vnc{{}^\fu \alpha, {\rm d}\fu} \Lambda.\]

\item If $N$ is of infinite type, then 
for any $C\in \Int{N}{M\vnc{\alpha}\Gamma}$,
there is a subgroup $\Lambda<\Gamma$ and an $\alpha$-$1$-cocycle $\fu \colon \Lambda \rightarrow M^{\rm u}$
with \[C = (C\cap M)\vnc{{}^\fu \alpha}\Lambda.\]
\end{enumerate}
\end{Thmint}
The proof is given in Section \ref{section:vn}.
We in fact show it for cocycle actions.
\section{Preliminaries}\label{section:Prelim}
\subsection{Notations}
Here we fix some notations used throughout the paper.
\begin{itemize}
\item The symbol `$\otimes$' is used for the minimal \Cs-tensor products and the Hilbert space tensor products.
\item For a normed space $X$, denote by $(X)_1$ the closed unit ball of $X$.
\item For $x, y\in X$ and $\epsilon>0$, we denote by $x \approx_\epsilon y$ when $\|x -y\|<\epsilon$.
For $x\in X$ and $S\subset X$, we denote by $x\in_\epsilon S$ when there is an element $s\in S$ with $x\approx_\epsilon s$.
\item For a \Cs-algebra $A$, denote by $A^{\rm p}$ the set of all projections in $A$.
Denote by $\cM(A)$ the multiplier algebra of $A$.
When $A$ is unital, we write $A^{\rm u}$ for the unitary group of $A$.
\item Put $\IB:=\IB(\ell^2(\IN))$, $\IK:=\IK(\ell^2(\IN))$
(the bounded operator algebra and the compact operator algebra on $\ell^2(\IN)$).
\item The unit of a group is usually denoted by $e$.
\item We use the symbol `$<$' for (not necessary proper) closed subgroups. So $H<G$ means that $H$ is a closed subgroup of a topological group $G$.
\end{itemize}
\subsection{The intermediate operator algebra lattice}
Given an inclusion of mathematical objects, one often has a natural lattice structure on the set of all intermediate objects.
In this paper, we study the \emph{intermediate operator algebra lattices}.

For an inclusion $B \subset A$ of \Cs-algebras, we write
$\Int{B}{A}$ for the set of all intermediate \Cs-algebras $B \subset C \subset A$.
It is not hard to check that $\Int{B}{A}$ forms a complete lattice with respect to the set inclusion order.
Indeed for a non-empty subset $S \subset \Int{B}{A}$, the supremum of $S$ in $\Int{B}{A}$ is
the \Cs-algebra generated by the union of $S$, while the infimum of $S$ is the intersection of $S$.

Similarly, for an inclusion $N\subset M$ of von Neumann algebras,
we write $\Int{N}{M}$ for the (complete) lattice of all intermediate von Neumann algebras $N \subset C \subset M$.
(Certainly $\Int{N}{M}$ is very different from the intermediate \Cs-algebra lattice of $N\subset M$.
However we do not distinguish the notation, because there is no danger of confusion in this article.
For von Neumann algebra inclusions, we only consider the intermediate von Neumann algebra lattice.)

\subsection{Compact group actions and eigenspaces}\label{subsection:compact}
For a compact group $K$, let $m_K$ denote its Haar probability measure.

Let $\alpha \colon K \acts A$ be an action of a compact abelian group on a \Cs-algebra.
For each character $\chi \in \widehat{K}$, define the \emph{eigenspace $A_\chi$ of $\alpha$ associated with $\chi$} to be
\[A_\chi:=\{a\in A: \alpha_s(a)=\chi_s \cdot a \quad{\rm~for~}s\in K\}.\]
For $\chi, \kappa\in \widehat{K}$, clearly one has
\[(A_\chi)^\ast=A_{\chi^{-1}},\quad A_\chi \cdot A_\kappa \subset A_{\chi \cdot \kappa}.\]
For each $\chi \in \widehat{K}$, one has a $K$-equivariant completely contractive projection
\[\E^K_{\chi}\colon A \rightarrow A_\chi\]
given by
\[\E^K_\chi(a):=\int_{K} \alpha_s(a)\overline{\chi_s} {\rm d}m_K(s).\]
It is known that \[\overline{\rm span}\left\{A_\chi: \chi \in \widehat{K}\right\}=A.\]
(See Section 8.1 in \cite{Pedbook} for instance.)
When $\chi=1$, $A_1$ is nothing but the fixed point algebra $A^K$.
We simply denote $\E^K_1$ by $\E^K$.
Note that $\E^K$ is a faithful conditional expectation.

Obviously the analogous definitions work for von Neumann algebras.
We use the same notations for the von Neumann algebra variants.
Note that in this case, the eigenspaces are weak-$\ast$ closed, and the projections $\E^K_\chi$ are normal.

\subsection{Cocycle actions and reduced twisted crossed products}
Let $A$ be a \Cs-algebra, and let $\Gamma$ be a discrete group.
Recall that a \emph{cocycle} (or \emph{twisted}) \emph{action} $(\alpha, \fw) \colon \Gamma \acts A$
is a pair of maps
\[\alpha \colon \Gamma \rightarrow {\rm Aut}(A), \quad \fw \colon \Gamma \times \Gamma \rightarrow \cM(A)^{\rm u}\]
satisfying the axioms
\[\alpha_s\circ \alpha_t=\ad(\fw_{s, t})\circ \alpha_{st}, \quad \alpha_r(\fw_{s, t})\fw_{r, st}=\fw_{r, s}\fw_{rs, t},\quad \fw_{e, s}=\fw_{s, e}=1\] 
for $r, s, t\in \Gamma$.
The map $\fw$ is referred to as a $2$-cocycle of $\alpha$.
Cocycle actions naturally appear in many contexts, e.g., in the studies of group extensions and corners of actions.
We refer the reader to \cite{Sut}, \cite{PR}, \cite{Wil} for some basic facts on cocycle actions.
In this article, cocycle actions play a central role to describe intermediate \Cs-algebras of the crossed product inclusions $B \subset A \rca{\alpha} \Gamma$
arising from $\Gamma$-\Cs-algebra inclusions $B \subset (A, \alpha)$.
This may be an unexpected and novel application of cocycle actions.

Similar to the ordinary actions, associated to a cocycle action $(\alpha, \fw) \colon \Gamma \acts A$,
one defines the \emph{algebraic twisted crossed product} $A\rtimes_{{\rm alg}, \alpha, \fw} \Gamma$ as follows.
First, as a linear space, it is the space of all finitely supported $A$-valued functions on $\Gamma$.
For $a\in A$ and $s\in \Gamma$, we write $a s^{\fw}$ for the function $a \delta_s$ in this algebra.
For ordinary actions (i.e., the case $\fw = 1$), we simply denote $s^\fw$ by $s$ for short.
The multiplication and involution on $A\rtimes_{{\rm alg}, \alpha, \fw} \Gamma$ are given by
\[(as^\fw)(bt^\fw):=a\alpha_s(b)\fw_{s, t} (st)^{\fw}, \quad (as^\fw)^\ast := \alpha_{s}^{-1}(a^\ast)\fw_{s^{-1}, s}^\ast (s^{-1})^\fw\]
for $a, b\in A,~ s, t\in \Gamma$.
We identify $A$ with the $\ast$-subalgebra $Ae^\fw$ of $A\rtimes_{{\rm alg}, \alpha, \fw} \Gamma$ in the obvious way.

The \emph{reduced twisted crossed product} $A\rca{\alpha, \fw} \Gamma$ of $(\alpha, \fw)$ is the unique \Cs-completion of $A\rtimes_{{\rm alg}, \alpha, \fw} \Gamma$
which admits a faithful conditional expectation $\E\colon A\rca{\alpha, \fw} \Gamma\rightarrow A$
with $\E(x)=x(e)$ for $x\in A\rtimes_{{\rm alg}, \alpha, \fw} \Gamma$.
Similar to the ordinary crossed product, the existence of such a completion is shown via regular covariant representations of $(\alpha, \fw)$,
and the uniqueness is then clear from the faithfulness condition.
Note that $s^\fw\in \cM(A\rca{\alpha, \fw} \Gamma)^{\rm u}$ for $s\in \Gamma$.

Recall that two cocyle actions $(\alpha, \fw), (\beta, \fx) \colon \Gamma \acts A$ are said to be \emph{exterior equivalent}
if there is a map $\fu \colon \Gamma \rightarrow \cM(A)^{\rm u}$ satisfying
\[\alpha_s=\ad(\fu_s)\circ \beta_s \quad{\rm~for~}s\in \Gamma, \quad \fw_{s, t}=\fu_s\beta_s(\fu_t)\fx_{s, t} \fu_{st}^\ast\]
for $s, t\in \Gamma$.
Note that when $(\alpha, \fw)$ and $(\beta, \fx)$ are exterior equivalent,
then their reduced crossed products are isomorphic.
Indeed the map $a s^{\fx} \mapsto a\fu_s s^\fw$ extends to the isomorphism
$A\rca{\beta, \fx} \Gamma \cong A\rca{\alpha, \fw}\Gamma$.

For $(\alpha, \fw)\colon \Gamma \acts A$,
the $2$-cocycle $\fw$ is said to be a \emph{coboundary} if 
$(\alpha, \fw)$ is exterior equivalent to a genuine action $(\beta, 1) \colon \Gamma \acts A$.

For $s\in \Gamma$, we define the completely contractive map
\[\E_s\colon A \rca{\alpha, \fw} \Gamma \rightarrow A\]
by the formula
\[\E_s(x):=\E(x (s^\fw)^\ast).\] 
Clearly $\E_s$ is a continuous extension of the evaluation map
\[{\rm ev}_s\colon A\rtimes_{{\rm alg}, \alpha, \fw} \Gamma \rightarrow A;\quad\quad x\mapsto x(s).\]
We use the same notations for the analogous maps on twisted crossed product von Neumann algebras.

\subsection{Twisted crossed product \Cs-subalgebras in crossed products}
The following maps are fundamental in our splitting theorems.
\begin{Def}\label{Def:resp}
Let $(A, \alpha)$ be a $\Gamma$-\Cs-algebra.
Let $C \subset A$ be a non-degenerate \Cs-subalgebra.
Let $\Lambda<\Gamma$ be a subgroup.
We say a map $\fu\colon \Lambda \rightarrow \cM(A)^{\rm u}$ is \emph{$(C, \Lambda)$-respecting} if it satisfies the relations
\[\fu_s \alpha_s(C) \fu_s^\ast =C, \quad \fu_s \alpha_s(\fu_t)\fu_{st}^\ast \in \cM(C)^{\rm u} \quad{\rm ~for~}s, t\in \Lambda.\]
\end{Def}
Associated to a $(C, \Lambda)$-respecting map $\fu$,
we define the maps
\[{}^\fu \alpha \colon \Lambda \rightarrow \Aut(C) \quad {\rm and}\quad {\rm d}\fu \colon \Lambda \times \Lambda \rightarrow \cM(C)^{\rm u}\]
to be
\[{}^\fu \alpha_s:=\ad(\fu_s)\circ \alpha_s|_C, \quad {\rm d}\fu_{s, t}:= \fu_s \alpha_s(\fu_t) \fu_{st}^\ast \quad {\rm~for~}s, t\in \Lambda.\]
Then the pair $({}^\fu \alpha, {\rm d}\fu)$ forms a cocycle action of $\Lambda$ on $C$.
We remark that ${\rm d}\fu$ may not be a coboundary in $C$,
because the image of $\fu$ may not sit in $\cM(C)$.

For a $(C, \Lambda)$-respecting map $\fu$, we next see that $C\rca{{}^\fu \alpha, {\rm d}\fu} \Lambda$
is identified with a \Cs-subalgebra of $A \rca{\alpha} \Gamma$ in a canonical and unique way.
Let $\iota \colon C \rightarrow A \rca{\alpha}\Gamma$ denote the inclusion map.
Define a map $\fv\colon \Lambda \rightarrow \cM(A \rca{\alpha}\Gamma)^{\rm u}$
to be $\fv_s:=\fu_s \cdot s$ for $s\in \Lambda$.
Then $(\iota, \fv)$ is a covariant representation of $({}^\fu \alpha, {\rm d}\fu)$.
Hence one has a $\ast$-homomorphism
\[\iota \rtimes \fv \colon C \rtimes_{{\rm alg}, {}^\fu \alpha, {\rm d}\fu}\Lambda \rightarrow A \rca{\alpha} \Gamma\]
given by sending $c s^{{\rm d}\fu}$ to $c\fv_s$ for $c\in C$ and $s\in \Lambda$.
The canonical conditional expectation $\E\colon A\rca{\alpha}\Gamma \rightarrow A$ restricts to a faithful conditional expectation from the closure of the image of $\iota \rtimes \fv$ onto $C$.
Moreover the restriction is an extension of the evaluation map ${\rm ev}_e$ on (the image of) $C\rtimes_{{\rm alg}, {}^\fu \alpha, {\rm d}\fu}\Lambda$.
Thus $\iota\rtimes \fv$
extends to an embedding $C\rca{{}^\fu \alpha, {\rm d}\fu} \Lambda \rightarrow A \rca{\alpha} \Gamma$.
Via this embedding, we identify $C\rca{{}^\fu \alpha, {\rm d}\fu} \Lambda$ with a \Cs-subalgebra of $A\rca{\alpha} \Gamma$.
We next see that this subalgebra depends only on $({}^\fu \alpha, {\rm d}\fu)$ (not on the choice of $\fu$).
This justifies the notation $C\rca{{}^\fu \alpha, {\rm d}\fu} \Lambda \subset A\rca{\alpha} \Gamma$.
In fact we observe a slightly more general result.

Let $C_i\subset A$, $\Lambda_i< \Gamma$, and $(C_i, \Lambda_i)$-respecting maps $\fu_i\colon \Lambda_i \rightarrow \cM(A)^{\rm u}$; $i=1, 2$, be given.
Then the two \Cs-subalgebras $C_i \rca{{}^{\fu_i} \alpha,{\rm d}\fu_i} \Lambda_i \subset A \rca{\alpha} \Gamma$; $i=1, 2$ coincide (as subsets) if and only if they satisfy
\[C_1=C_2,\quad \Lambda_1=\Lambda_2,\quad {\rm~and~} \quad \fu_{1, s} \fu_{2, s}^\ast \in \cM(C_1)^{\rm u} \quad{\rm~ for~} s\in \Lambda_1.\]
Indeed if they satisfy the conditions,
then the pointwise product $\fu_1 \fu_2^\ast$ witnesses the exterior equivalence, while the converse is clear.
In other words, two cocycle actions of the above form give rise to the same \Cs-subalgebra
if and only if they are exterior equivalent.
In particular, the subalgebra $C\rca{{}^\fu \alpha, {\rm d}\fu} \Lambda \subset A \rca{\alpha} \Gamma$
depends only on the cocycle action $({}^\fu \alpha, {\rm d}\fu)$, and does not depend on the choice of $\fu$.

More generally, for $C_i$, $\Lambda_i$, $\fu_i$; $i=1, 2$, as above,
one has $C_1 \rca{{}^{\fu_1}\alpha, {\rm d}\fu_1} \Lambda_1 \subset C_2 \rca{{}^{\fu_2}\alpha, {\rm d}\fu_2} \Lambda_2$ in $A \rca{\alpha} \Gamma$ if and only if
they satisfy
\[C_1 \subset C_2,\quad \Lambda_1 < \Lambda_2, \quad \fu_{1, s}\fu_{2, s} ^\ast \in \cM(C_2)^{\rm u} \quad {\rm~for~}s\in \Lambda_1.\]

We give a remark on our notations. When $C\rca{{}^\fu \alpha, {\rm d}\fu} \Lambda$ is regarded as a \Cs-subalgebra of $A \rca{\alpha} \Gamma$ in the above way,
even on $C\rca{{}^\fu \alpha, {\rm d}\fu} \Lambda$, we write $\E_s$ for (the restriction of) the canonical projection of $A \rca{\alpha} \Gamma$,
not that of $C\rca{{}^\fu \alpha, {\rm d}\fu} \Lambda$, for $s\in \Gamma$.

Clearly the analogous definitions and the above statements work for von Neumann algebras.
We use the obvious notations for them.

\subsection{Characterizations of twisted crossed product \Cs-subalgebras}
Haagerup and Kraus \cite{HK} introduced a moderate finite approximation property of locally compact groups so called the \emph{approximation property} (AP).
The AP is weaker than the weak amenability, stronger than the exactness (see \cite{SuzAP} for the non-discrete case), and is closed under many manipulations of groups, e.g., extensions, increasing unions, free products.
Hence many interesting groups, including amenable groups, hyperbolic groups (\cite{Ozwk}), rank one lattices, and their extensions and free products have the AP.
We refer the reader to the original article \cite{HK} and Section 12.4 of the book \cite{BO} for details.

Similar to the previous works \cite{Suz17}, \cite{SuzCMP}, we use the AP to characterize \Cs-subalgebras of the reduced crossed products via the coefficients.
In fact the next result is the only point where we need the AP of the acting group in Theorem \ref{Thmint:Cs}.
\begin{Lem}\label{Lem:AP}
Let $(A, \alpha)$ be a $\Gamma$-\Cs-algebra.
Let $C\subset A$ be a non-degenerate \Cs-subalgebra, let $\Lambda<\Gamma$,
and
let $\fu\colon \Lambda \rightarrow \cM(A)^{\rm u}$ be a $(C, \Lambda)$-respecting map.
Assume that $\Lambda$ has the AP.
Then $x\in A \rca{\alpha} \Gamma$ is contained in $C \rca{{{}^{\fu}\alpha}, {\rm d}\fu} \Lambda$ if and only if it satisfies
$\E_s(x)\in C \fu_s$ for all $s\in \Lambda$ and $\E_s(x)=0$ for all $s\in \Gamma \setminus \Lambda$.
\end{Lem}
\begin{proof}
This follows from the proof of \cite{Suz17}, Proposition 3.4.
\end{proof}

We note that the statement fails for non-exact groups,
even in the context of the crossed product splitting theorem (see Proposition 2.7 in \cite{SuzCMP}).
It would be a deep question if the statement could still be true for some/all exact groups without the AP (e.g., $\SL(3, \IZ)$ \cite{LS}).

In the von Neumann algebra case, thanks to the von Neumann bicommutant theorem,
the analogous statement holds for all discrete groups. See Lemma \ref{Lem:coeffv}.

For a \Cs-algebra $A$ and a set $S$, we set
\[\ell^2(S, A):=\left\{(a_s)_{s\in S}\in A^S\colon \sum_{s\in S} a_s^\ast a_s ~\text{converges in norm}\right\}.\]
We regard it as a \Cs-correspondence over $A$ in the standard way.

Now we recall the definition of central freeness from \cite{SuzCMP}.
\begin{Def}[Central freeness of inclusions for automorphisms and actions]\label{Def:cf}
Let $B \subset A$ be a non-degenerate \Cs-algebra inclusion.
Let $\alpha$ be an automorphism of $A$ with $\alpha(B)=B$.
We say that $\alpha$ is \emph{centrally free for} $B\subset A$ if it satisfies the following condition:
For any $\epsilon>0$ and any $a, b \in A$, there is $\xi\in (\ell^2(\IN, B))_1$ with
\[\sum_{n\in \IN} \xi_n^\ast a \xi_n \approx_\epsilon a, \quad \sum_{n\in \IN} \xi_n^\ast b \alpha(\xi_n) \approx_\epsilon 0.\]

A $\Gamma$-\Cs-algebra inclusion $B \subset (A, \alpha)$ is said to be \emph{centrally $\Gamma$-free} \cite{SuzCMP} if
the automorphisms $\alpha_s$; $s\in \Gamma \setminus \{e\}$, are centrally free for $B\subset A$.
\end{Def}
We note that when $B=A$, central freeness of $\alpha \in {\rm Aut}(A)$
is weaker than the Rohlin type properties \cite{Nak}, \cite{IzuR} and
the isometric shift-absorption \cite{GS2}.

\subsection{Irreducibility and \Cs-irreducibility for operator algebra inclusions}
Recall that a non-degenerate inclusion $B \subset A$ of \Cs-algebras
is called \emph{irreducible} if the relative commutant $B' \cap \cM(A)$ is equal to the center of $\cM(A)$.
For subfactors $N\subset M$,
this condition is equivalent to the factoriality of all members of $\Int{N}{M}$.
Hence irreducibility plays the role of the simplicity in subfactor theory.

For \Cs-algebras, as the triviality of the center is rather moderate condition,
the above irreducibility condition is not sufficient for many purposes.
Recently R{\o}rdam \cite{RorIrr} introduced
the following stronger version of irreducibility.
A \Cs-algebra inclusion $B\subset A$ is said to be \emph{\Cs-irreducible}
if all members of $\Int{B}{A}$ are simple.
As many evidences are investigated in \cite{RorIrr},
\Cs-irreducibility will be the right notion of simplicity for inclusions of simple \Cs-algebras.
For basic examples of \Cs-irreducible inclusions, see \cite{RorIrr} and \cite{Mu} for instance.

\section{Proof of Theorem \ref{Thmint:Cs}}\label{section:proof}
We split the main parts of the proof into a few lemmas.
The first lemma is the point where central freeness plays a fundamental role (cf.~\cite{SuzCMP} for related applications).
\begin{Lem}\label{Lem:cf}
Assume that $B \subset (A, \alpha)$ is centrally $\Gamma$-free.
Then, for any $s\in \Gamma$ and any 
$C\in \Int{B}{A \rca{\alpha} \Gamma}$,
one has $\E_s(C) s \subset C$.
\end{Lem}
\begin{proof}
Let $C\in \Int{B}{A \rca{\alpha} \Gamma}$.
Let $c\in C$ and $s\in \Gamma$ be given.
We need to show that $\E_s(c)s\in C$.
Since $C$ is closed, it suffices to show the next statement:
For any $\epsilon>0$, $\E_s(c)s\in_{4\epsilon} C$.
We show this claim. 

Take any positive number $\epsilon$.
Choose $\tilde{c}_0 \in A\rtimes_{{\rm alg}, \alpha} \Gamma$ with $\tilde{c}_0 \approx_\epsilon c$.
We write
\[\tilde{c}_0=\sum_{t\in F} a^{(0)}_t ts,\]
where $F=\{e, s_1, \ldots, s_k\} \subset \Lambda$ is a finite subset (with $|F|=k+1$) and $a^{(0)}_t\in A$ for each $t\in F$.

In the case $k=0$, clearly we have $\E_s(c) s \approx_{2\epsilon} c\in C$ and hence the statement holds.
Assume that $k\geq 1$.
Since $B \subset (A, \alpha)$ is centrally $\Gamma$-free, one can find $\xi^{(1)} \in (\ell^2(\IN, B))_1$ with
\[\sum_{n\in \IN} (\xi_n^{(1)})^\ast a^{(0)}_e \xi_n^{(1)} \approx_{\epsilon/k} a^{(0)}_e,\quad \sum_{n\in \IN} (\xi_n^{(1)})^\ast a^{(0)}_{s_1} \alpha_{s_1}(\xi_n^{(1)}) \approx_{\epsilon/k} 0.\]
The second relation implies that
\[\sum_{n\in \IN} (\xi_n^{(1)})^\ast \tilde{c}_0 \alpha_{s}^{-1}(\xi_n^{(1)}) \approx_{\epsilon/k} \sum_{t\in F\setminus\{s_1\}} a^{(1)}_t ts =:\tilde{c}_1,\]
where
\[a^{(1)}_t:=\sum_{n\in \IN} (\xi_n^{(1)})^\ast a^{(0)}_t \alpha_{t}(\xi_n^{(1)})\in A \quad {\rm~for~}t\in F\setminus \{s_1\}.\]
Note that $a^{(1)}_e \approx_{\epsilon/k} a^{(0)}_e$ by the first relation on $\xi^{(1)}$.

When $k=1$, we finish the process.
Otherwise, we next apply the same argument to $\tilde{c}_1$ and $s_2$ instead of $\tilde{c}_0$ and $s_1$.
Then we obtain $\xi^{(2)}\in (\ell^2(\IN, B))_1$ with
\[ \sum_{n\in \IN}( \xi_n^{(2)})^\ast \tilde{c}_1 \alpha_{s}^{-1}(\xi_n^{(2)}) \approx_{\epsilon/k} \sum_{t\in F\setminus \{s_1, s_2\} } a^{(2)}_{t} ts=:\tilde{c}_2\]
for some $a^{(2)}_t \in A$; $t\in F\setminus\{s_1, s_2\}$, with $a^{(2)}_e \approx_{\epsilon/k} a^{(1)}_e$.
By the Cauchy--Schwarz inequality for right Hilbert \Cs-modules, the function $\zeta \in (\ell^2(\IN^2, B))_1$ given by
$\zeta_{n, m}:=\xi_n^{(1)}\xi_m^{(2)}$ satisfies
\[\sum_{n, m\in \IN} \zeta_{n, m}^\ast \tilde{c}_0 \alpha_s^{-1}(\zeta_{n, m}) \approx_{2\epsilon/k} \tilde{c}_2.\]
We iterate this process $k$ times. Then, at the end,
we obtain $\xi \in (\ell^2(\IN^k, B))_1$ and $a^{(k)}_e\in A$
with 
\[\sum_{\n\in \IN^k} \xi_\n^\ast\tilde{c}_0 \alpha_s^{-1}(\xi_\n) \approx_\epsilon a^{(k)}_e s \approx_\epsilon a^{(0)}_e s \approx_\epsilon \E_s(c)s.\]
The Cauchy--Schwarz inequality implies
\[\sum_{\n\in \IN^k} \xi_\n^\ast \tilde{c}_0 \alpha_s^{-1}(\xi_\n) \approx_\epsilon \sum_{\n\in \IN^k} \xi_\n^\ast c \alpha_s^{-1}(\xi_\n).\]
Since $B \subset C$ and $c\in C$, the last element is contained in $C$.
Hence $\E_s(c)s\in_{4\epsilon} C$.
\end{proof}
By applying Lemma \ref{Lem:cf} to $s=e$, we obtain the \Cs-irreducibility result.
(For this result, we do not need the AP of the acting group.)
\begin{Cor}\label{Cor:Cirr}
Let $B \subset (A, \alpha)$ be a \Cs-irreducible and centrally $\Gamma$-free inclusion of $\Gamma$-\Cs-algebras.
Then the inclusion $B \subset A\rca{\alpha}\Gamma$ is \Cs-irreducible.
\end{Cor}

The next lemma explains the role of the new tensor component $\mathcal{O}_2$.
As we see in Section \ref{section:O2}, due to obstructions on K$_0$-groups, in general this operation is really necessary to obtain the clean splitting.
(However see also Remark \ref{Rem:partial} for a partial splitting result before applying $\mathcal{O}_2\otimes -$.)

We say a \Cs-algebra inclusion $B \subset A$ is \emph{$\mathcal{O}_2$-absorbing} if
all members $C$ of $\Int{B}{A}$ satisfy $C \cong \mathcal{O}_2 \otimes C$.
\begin{Lem}\label{Lem:unitary}
Let $B \subset (A, \alpha)$ be a unital $\Gamma$-\Cs-algebra inclusion that is \Cs-irreducible, centrally $\Gamma$-free, and $\mathcal{O}_2$-absorbing.
Then, for any
$C\in \Int{B}{A \rca{\alpha} \Gamma}$ and any $s\in \Gamma$, one has
either $\E_s(C) = \{0\}$ or $\E_s(C) \cap A^{\rm u} \neq \emptyset$.
Moreover, in the latter case, for any $u\in \E_s(C) \cap A^{\rm u}$, one has
$\E_s(C)=(C\cap A) \cdot u$.
\end{Lem}
\begin{proof}
Let $s \in \Gamma$ be an element with $\E_s(C) \neq \0$.
We first show that $\E_s(C)\cap A^{\rm u} \neq \emptyset$.
Take $c\in C$ with $\E_s(c) \neq 0$.
By Lemma \ref{Lem:cf}, one has $\E_s(c) s \in C$. Hence $\E_s(c) \E_s(c)^\ast \in C \cap A$.
By the assumptions on $B \subset A$, every $D\in \Int{B}{A}$ is purely infinite simple with K$_0(D)=0$.
Hence one can find $x\in C \cap A$ satisfying
\[x \E_s(c)\E_s(c)^\ast x^\ast =1.\]
This implies that $p:=\E_s(c)^\ast x^\ast x \E_s(c)$ is a projection.
Because $x\E_s(c)s\in C$, one has $p\in \alpha_s(C \cap A)$.
By \cite{Cun},
one can find $v\in \alpha_s(C \cap A)$
with $v^\ast v=1$, $vv^\ast =p$.
Then one has $x \E_s(c)v s =x \E_s(c) s \alpha_{s}^{-1}(v) \in C$.
Hence $x \E_s(c)v \in \E_s(C)\cap A^{\rm u}$.

For the last statement, take any $u\in \E_s(C) \cap A^{\rm u}$.
Then for any $x\in \E_s(C)$, as $xs, us\in C$ by Lemma \ref{Lem:cf}, one has $xu^\ast \in C \cap A$.
Hence $x=(xu^\ast)\cdot u\in (C \cap A) \cdot u$.
The reverse inclusion is trivial, because $\E_s$ is left $A$-linear.
\end{proof}
Now combining the lemmas, we complete the proof in the unital case.
\begin{proof}[Proof of Theorem \ref{Thmint:Cs}]
Clearly it suffices to show the statement under the additional assumption that the inclusion $B\subset (A, \alpha)$ is $\mathcal{O}_2$-absorbing.
We put
\[D:=C \cap A, \quad \Lambda:=\{s\in \Gamma : \E_s(C) \neq \0\}.\]
Lemmas \ref{Lem:cf} and \ref{Lem:unitary} immediately imply that $\Lambda$ forms a subgroup of $\Gamma$.
Moreover these lemmas allow us to choose a map $\fu \colon \Lambda \rightarrow C^{\rm u}$ with
\[\E_s(C)s=D \fu_s s = C \cap (As) \quad{\rm~ for~ all~} s\in \Lambda.\]
These conditions imply that
$\fu$ is $(D, \Lambda)$-respecting.
It is now clear from Lemma \ref{Lem:AP} that
\[C=D \rca{{}^\fu \alpha, {\rm d}\fu} \Lambda.\]
\end{proof}
\begin{Rem}
Clearly, for the conclusion of Theorem \ref{Thmint:Cs}, we only need to assume
the central $\Gamma$-freeness of the tensor product inclusion $\mathcal{O}_2 \otimes B \subset (\mathcal{O}_2 \otimes A, \id_{\mathcal{O}_2} \otimes \alpha)$
rather than that of the original inclusion.
The former condition is often easier to check, thanks to classification theory of \Cs-dynamical systems.
For instance, when $A$ is simple, serapable, and nuclear, for any pointwise outer action $\alpha \colon \Gamma \acts A$,
the trivial inclusion $\mathcal{O}_2 \otimes A \subset (\mathcal{O}_2 \otimes A, \id_{\mathcal{O}_2} \otimes \alpha)$
is centrally $\Gamma$-free
by the Rohlin type results for cyclic groups \cite{Nak} and \cite{IzuR}, or by Proposition 3.15 of \cite{GS2}.

In particular, for any such $(A, \alpha)$ and any non-degenerate \Cs-irreducible $\Gamma$-\Cs-algebra inclusion $C\subset (B, \beta)$, Theorem \ref{Thmint:Cs} holds for $A\otimes C \subset (A\otimes B, \alpha \otimes \beta)$.
(See Theorem \ref{Thm:generalCs} for the non-unital case.)
\end{Rem}
\begin{Rem}\label{Rem:partial}
Under the same assumptions as those of Theorem \ref{Thmint:Cs},
without adding the tensor component $\mathcal{O}_2$, we still have the following crossed product-like splitting of intermediate \Cs-algebras as follows.
For $C\in \Int{B}{A \rca{\alpha} \Gamma}$,
set
\[\Lambda:=\{s\in \Gamma: \E_s(C) \neq \{0\}\}.\]
Because the subset $\Lambda$ is not changed after taking the tensor product with $\mathcal{O}_2$,
it also forms a subgroup of $\Gamma$.
(This can also be shown directly by using the \Cs-irreducibility of $B\subset A$,
but we do not use this fact.)
For $s\in \Lambda$,
set $X_s:=\E_s(C)$. Then, as $X_s s=C \cap (A s)$ for $s\in \Lambda$ by Lemma \ref{Lem:cf},
each $X_s$ is a closed subspace of $A$.
The indexed family $\mathfrak{X}_C:=(X_s)_{s\in \Lambda}$ satisfies the axioms
\[X_s^\ast =\alpha_{s}^{-1}(X_{s^{-1}}), \quad X_s \alpha_s(X_t)\subset X_{st} \quad {\rm~for~all~}s, t\in \Lambda,\quad {\rm and}\quad B \subset X_e.\]
By the proof of Proposition 3.4 in \cite{Suz17}, one has
\[C=\{x\in A \rca{\alpha}\Lambda: \E_s(x)\in X_s \quad{\rm~for~}s\in \Lambda\}.\]
Conversely, assume that one has $\Lambda<\Gamma$ and an indexed family $\mathfrak{X}:=(X_s)_{s\in \Lambda}$
of nonzero closed subspaces of $A$
satisfying the above axioms.
Then this defines $C_\mathfrak{X} \in \Int{B}{A \rca{\alpha} \Gamma}$ by the same formula
\[C_\mathfrak{X}:=\{x\in A \rca{\alpha}\Lambda: \E_s(x)\in X_s \quad{\rm~for~}s\in \Lambda\}.\]
Again by the proof of Proposition 3.4 in \cite{Suz17}, one has
\[C_{\mathfrak{X}_C}=C, \quad \mathfrak{X}_{C_\mathfrak{X}}=\mathfrak{X}\]
for all these $C$ and $\mathfrak{X}$.
Thus the map $\Int{B}{A \rca{\alpha} \Gamma}\ni C \mapsto \mathfrak{X}_C$
gives a direct and complete classification of the intermediate \Cs-algebras by the indexed families $\mathfrak{X}$ of nonzero closed subspaces satisfying the axioms.
The \Cs-algebra $C_\mathfrak{X}$ may be seen as the reduced \Cs-algebra of a twisted version of \emph{partial actions} of $\Lambda$ (see e.g., \cite{Exe}),
but we do not discuss it further in this article.

This partial splitting still holds true without \Cs-irreducibility,
but in this case, the subset $\Lambda$ needs not form a subgroup.
\end{Rem}

\section{Twisted and/or non-unital case}\label{section:general}
In this section, we extend Theorem \ref{Thmint:Cs} to cocycle and/or non-unital actions.
Here we deduce the non-unital case from the unital cocycle case.
The proof for unital cocycle actions is essentially the same as the unital ordinary action case.
Therefore we only explain how to modify the original proof for this generalization.

First let us extend some notations for the generality in the present section.
Let $(\alpha, \fw) \colon \Gamma \acts A$ be a cocycle action.
Consider a non-degenerate \Cs-subalgebra $B \subset A$ and a subgroup $\Lambda<\Gamma$ satisfying
\[\alpha_s(B)=B ,\quad \fw_{s, t}\in \cM(B)^{\rm u} \quad {\rm~for~}s, t\in \Lambda.\]
Then $(\alpha, \fw)$ restricts to a cocycle action $\Lambda\acts B$.
(We write it by the same symbol $(\alpha, \fw)$ for short.)
We refer to $B$ as a \emph{$\Lambda$-\Cs-subalgebra} of $(A, \alpha, \fw)$.
When $B$ is a $\Gamma$-\Cs-subalgebra, we write $B \subset (A, \alpha, \fw)$.

We say that a $\Gamma$-\Cs-algebra inclusion $B \subset (A, \alpha, \fw)$ is \emph{centrally $\Gamma$-free} if
the automorphisms $\alpha_s$; $s\in \Gamma \setminus \{e\}$, are centrally free for $B \subset A$ (see Definition \ref{Def:cf}).

Let $C\subset A$ be a non-degenerate \Cs-subalgebra and $\Lambda<\Gamma$ be a subgroup.
We say that a map $\fu \colon \Lambda \rightarrow \cM(A)^{\rm u}$ is \emph{$(C, \Lambda)$-respecting}, if it satisfies
\[\fu_s\alpha_s(C) \fu_s^\ast =C, \quad \fu_s \alpha_s(\fu_t)\fw_{s, t} \fu_{st}^\ast \in \cM(C)^{\rm u} \quad {\rm~for~}s, t\in \Lambda.\]
For a $(C, \Lambda)$-respecting map $\fu$, the maps
\[{{}^{\fu}\alpha}\colon \Lambda \rightarrow \Aut(C), \quad {}^{\fu}\fw\colon \Lambda \times \Lambda \rightarrow \cM(C)^{\rm u}\]
given by
\[{{}^{\fu}\alpha}_s:=\ad(\fu_s)\circ \alpha_s|_{C},\quad {}^{\fu}\fw_{s, t}:= \fu_s \alpha_s(\fu_t)\fw_{s, t} \fu_{st}^\ast \in \cM(C)^{\rm u} \quad {\rm~for~}s, t\in \Lambda\]
define a cocycle action of $\Lambda$ on $C$.
Similar to the ordinary action case (see Section \ref{section:Prelim}), we identify
$C \rca{{{}^{\fu}\alpha}, {}^{\fu}\fw} \Lambda$ with a \Cs-subalgebra of $A \rca{\alpha, \fw} \Gamma$ in the canonical way.

For a cocycle action $(\alpha, \fw) \colon \Gamma \acts A$ and $B\subset (A, \alpha, \fw)$,
consider any map $\fu\colon \Gamma \rightarrow \cM(B)^{\rm u}$.
Then, for any $\Lambda< \Gamma$, any $C\in \Int{B}{A}$, and any $(C, \Lambda)$-respecting map $\fv$,
the inclusions
\[C\rca{{}^{\fv\fu^\ast}\hspace{-1mm}(^{\fu}\alpha), {}^{\fv\fu^\ast}\hspace{-1mm}({}^{\fu}\fw)}\Lambda \subset A \rca{{{}^{\fu}\alpha}, {}^{\fu}\fw} \Gamma \quad{\rm~and~}\quad
C\rca{{}^\fv \alpha, {}^\fv\fw}\Lambda \subset A \rca{\alpha, \fw}\Gamma\]
are isomorphic.
(Here we note that the pointwise multiplication $\fv\fu^\ast$ is $(C, \Lambda)$-respecting with respect to $({}^\fu \alpha, {}^\fu \fw) \colon \Gamma \acts A$,
and hence the notations make sense.)
Indeed the isomorphism $\theta \colon A \rca{\alpha, \fw}\Gamma \rightarrow A\rca{{{}^{\fu}\alpha}, {}^{\fu}\fw} \Gamma$
given by $\theta(as^{\fw})=a\fu_s^\ast s^{{}^{\fu}\fw}$ satisfies
\[\theta(C\rca{{}^\fv \alpha, {}^\fv\fw}\Lambda)=C\rca{{}^{\fv\fu^\ast}\hspace{-1mm}({{}^{\fu}\alpha}), {}^{\fv\fu^\ast}\hspace{-1mm}({}^{\fu}\fw)}\Lambda\]
for all triplets $C, \Lambda, \fv$.
We apply this observation with the following version of the stabilization trick (\cite{Sut}, \cite{PR}).
\begin{Lem}\label{Lem:stab}
Let $(\alpha, \fw) \colon \Gamma \acts A$ be a cocycle action
and let $B \subset (A, \alpha, \fw)$.
Then there is a map $\fu \colon \Gamma \rightarrow \cM(B\otimes \IK)^{\rm u}$ satisfying the relation ${}^{\fu}(\fw\otimes 1_{\IK})=1$.
\end{Lem}
\begin{proof}
Note that the assumption implies $\fw_{s, t}\in \cM(B)^{\rm u}$ for $s, t\in \Gamma$.
Thus the proofs of the stabilization trick in \cite{Sut}, \cite{PR} give the desired map; see the formula (3.2) (and Lemma 3.6) in \cite{PR} for instance.
\end{proof}

By combining Lemmas \ref{Lem:AP} and \ref{Lem:stab} (and Lemma 5.2 of \cite{SuzCMP}),
we obtain the following generalization of Lemma \ref{Lem:AP}.

\begin{Lem}\label{Lem:AP2}
Let $(\alpha, \fw) \colon \Gamma \acts A$ be a cocycle action.
Let $C\subset A$ be a non-degenerate \Cs-subalgebra and $\Lambda<\Gamma$.
Let $\fu\colon \Lambda \rightarrow \cM(A)^{\rm u}$ be a $(C, \Lambda)$-respecting map.
Assume that $\Lambda$ has the AP.
Then $x\in A \rca{\alpha, \fw} \Gamma$ is contained in $C \rca{{{}^{\fu}\alpha}, {}^\fu \fw} \Lambda$ if and only if it satisfies
$\E_s(x)\in C \fu_s$ for all $s\in \Lambda$ and $\E_s(x)=0$ for all $s\in \Gamma \setminus \Lambda$.
\end{Lem}

We first extend Theorem \ref{Thmint:Cs} to unital cocycle actions.
\begin{Prop}\label{Prop:unital}
Let $\Gamma$ be a group with the AP.
Let $(\alpha, \fw) \colon \Gamma \acts A$ be a cocycle action on a unital \Cs-algebra.
Assume that $B \subset (A, \alpha, \fw)$ is \Cs-irreducible, centrally $\Gamma$-free, and $\mathcal{O}_2$-absorbing.
Then for any $C\in \Int{B}{A\rca{\alpha, \fw} \Gamma}$, there is $\Lambda< \Gamma$
and a $(C\cap A, \Lambda)$-respecting map $\fu \colon \Lambda \rightarrow A^{\rm u}$ with
\[C=(C\cap A)\rca{{}^\fu \alpha, {}^\fu \fw}\Lambda.\]
\end{Prop}
\begin{proof}
We observe that the proof of Lemma \ref{Lem:cf} works in the present situation, after replacing
$s$, $ts$ appearing in the series expansions of $\tilde{c}_i$ therein by $s^\fw$, $t^\fw s^\fw$ (not $(ts)^\fw$) respectively.
Then Lemma \ref{Lem:unitary} also works in this setting.
By using Lemma \ref{Lem:AP2} instead of Lemma \ref{Lem:AP},
the proof is complete.
\end{proof}

For any simple \Cs-algebra $A$, the tensor product $\mathcal{O}_2 \otimes A$ is purely infinite simple.
In particular it contains nonzero projections, and all nonzero projections are mutually Murray--von Neumann equivalent.
Moreover, when $A$ is non-unital and separable, for any nonzero $p\in (\mathcal{O}_2 \otimes A)^{\rm p}$, one can find a sequence $(v_n)_{n\in \IN}$ in $\mathcal{O}_2\otimes A$
with
\[v_n^\ast v_n=p \quad {\rm~for~all~}n\in \IN, \quad \sum_{n\in \IN} v_n v_n^\ast =1\quad{\rm~in~the~strict~topology}.\]
We refer the reader to the book \cite{Rorbook} for basic facts on purely infinite simple \Cs-algebras.

Building on these observations, we now complete the proof in full generality.
\begin{Thm}\label{Thm:generalCs}
Let $B \subset (A, \alpha, \fw)$ be a \Cs-irreducible, centrally $\Gamma$-free inclusion of separable $\Gamma$-\Cs-algebras.
Let $\Gamma$ be a group with the AP.
Then for any $C\in \Int{B}{A\rca{\alpha, \fw}\Gamma}$, there is $\Lambda<\Gamma$
and an $(\mathcal{O}_2\otimes (C\cap A), \Lambda)$-respecting map $\fu$ satisfying
\[\mathcal{O}_2\otimes C=(\mathcal{O}_2\otimes (C\cap A))\rca{{}^\fu \alpha, {}^\fu \fw}\Lambda.\]
\end{Thm}
\begin{proof}
We may assume that $A$ is non-unital and that the inclusion $B \subset (A, \alpha, \fw)$ is $\mathcal{O}_2$-absorbing.
We fix $p\in B^{\rm p}\setminus \{0\}$.
Choose a map $\fv \colon \Gamma \rightarrow B$ with
\[\fv_s\fv_s^\ast = p,\quad \fv_s^\ast \fv_s = \alpha_s(p)\]
for all $s\in \Gamma$.
Consider the cocycle action $(\alpha^p, \fw^p)\colon \Gamma \acts pAp$ defined by
\[\alpha^p_s(a):=\fv_s \alpha_s(a)\fv_s^\ast, \quad \fw^p_{s, t}:=\fv_s\alpha_s(\fv_t)\fw_{s, t} \fv_{st}^\ast\]
for $a\in pAp$ and $s, t\in \Gamma$.
Note that $pBp \subset (pAp, \alpha^p, \fw^p)$, and the inclusion is centrally $\Gamma$-free because the image of $\fv$ sits in $B$.
Observe that we have an isomorphism
$\Theta \colon (pAp)\rca{\alpha^p, \fw^p}\Gamma \rightarrow p(A\rca{\alpha, \fw}\Gamma) p$
satisfying
\[\Theta(pap)=pap\quad \text{ for } a\in A,\quad \Theta(s^{\fw^p})=\fv_s s^\fw \text{ for } s\in \Gamma.\]

Now let $C\in \Int{B}{A\rca{\alpha, \fw}\Gamma}$.
To complete the proof, it suffices to show that for any $s\in \Gamma$,
either $\E_s(C)=\0$ or there exists $\fu_s\in \cM(A)^{\rm u}$ with
$\E_s(C)=(C\cap A) \cdot \fu_s$.
Indeed if this claim holds, then the set
$\Lambda:=\{s\in \Gamma: \E_s(C)\neq \0\}$ forms a subgroup of $\Gamma$;
any map $\fu \colon \Lambda \rightarrow \cM(A)^{\rm u}$
satisfying $\E_s(C)=(C\cap A) \cdot \fu_s$ for all $s\in \Lambda$ is $(C\cap A, \Lambda)$-respecting;
and $C=(C\cap A) \rca{{}^\fu \alpha, {}^\fu \fw}\Lambda$ by Lemma \ref{Lem:AP2}.

To prove the claim, take any $s\in \Gamma$ with $\E_s(C)\neq \0$.
Put
\[D:=\Theta^{-1}(pCp)\in \Int{pBp}{(pAp)\rca{\alpha^p, \fw^p}\Gamma}.\]
By the definition of $\Theta$, we have
$\E_s(D)\neq \0$.
Hence by Proposition \ref{Prop:unital}, one can find $u\in (pAp)^{\rm u}$ with
$\E_s(D)=(D \cap (pAp))\cdot u$.
Note that $D\cap (pAp)= p(C \cap A)p$.
Hence we obtain
\begin{align*}p\E_s(C)\alpha_s(p)&=\E_s(pCp)\\
&=\E_s(D)\fv_s\\
&=p(C \cap A)u\fv_s.
\end{align*}

Now pick a sequence $(w_n)_{n\in \IN}$ in $B$
with $w_n^\ast w_n=p$ for $n\in \IN$ and $\sum_{n\in \IN}w_n w_n^\ast =1$ in the strict topology.
Then the series
\[\fu_s:=\sum_{n\in \IN} w_n u\fv_s \alpha_s(w_n)^\ast\]
strictly converges to an element in $\cM(A)^{\rm u}$.
Observe that, as $(w_n)_{n\in \IN} \subset B \subset C$, one has

\begin{align*}
w_nw_n^\ast \E_s(C) \alpha_s(w_m w_m^\ast) &=w_n \E_s(w_n^\ast C w_m)\alpha_s(w_m^\ast)\\
&=w_n \E_s(pCp)\alpha_s(w_m^\ast)\\
&=w_np\E_s(C)\alpha_s(p w_m^\ast)\\
&=w_np(C \cap A) w_m^\ast w_mu\fv_s \alpha_s(w_m^\ast)\\
&=w_np(C \cap A) w_m w_m^\ast w_mu\fv_s \alpha_s(w_m^\ast)\\
&=w_nw_n^\ast (C \cap A) w_mw_m^\ast \fu_s
\end{align*}
for $n, m\in \IN$.
By the same proof as that of Lemma \ref{Lem:cf}, we have
$\E_s(C)=A \cap (C(s^\fw)^\ast)$ and hence it is norm closed.
Thus the condition $\sum_{n\in \IN}w_n w_n^\ast =1$ implies
\[\E_s(C)=(C\cap A) \cdot \fu_s.\]
\end{proof}
\begin{Rem}
If $A$ is non-unital, simple, and separable,
then $\mathcal{O}_2 \otimes A$ is stable by \cite{Zha}.
As a result, by the stabilization trick \cite{Sut} (cf.~ Remark 1.5 in \cite{GS2}),
the $2$-cocycles in Theorem \ref{Thm:generalCs} can be removed in the non-unital case.
\end{Rem}

\section{Cocycle actions which are very hard to untwist}\label{section:nonexact}
In this section, for non-exact groups, we construct a cocycle action on $\mathcal{O}_2$ which is very hard to untwist.
(For the precise meaning, see Theorem \ref{Thm:untwist} below.)
Remark \ref{Rem:ambient} and Theorem \ref{Thm:untwist} below show that the $2$-cocycle appearing in our splitting theorem
cannot be removed even further tensoring with any unital cocycle action.

Only in this section, we treat \emph{locally compact} second countable (lcsc) groups $G$.
This is because the result in this section works for lcsc groups, and it is of independent interest.

Recall that a cocycle action $(\alpha, \fw) \colon G \acts A$ of an lcsc group on a separable \Cs-algebra
is a pair of Borel maps
\[\alpha\colon G \rightarrow \Aut(A),\quad \fw\colon G\times G \rightarrow \cM(A)^{\rm u}\]
satisfying the same axioms as the discrete group case.
Here, for a separable \Cs-algebra $A$, we equip ${\rm Aut}(A)$, $A$, $\cM(A)^{\rm u}$
with the point-norm topology, norm topology, strict topology, respectively.
Note that all these topologies are Polish, hence they are standard as Borel spaces.
Similarly, we require Borel measurability for the maps witnessing exterior equivalence.

It is fundamental that we only require Borel measurability of the maps.
Indeed, for two main motivating examples, namely cocycle actions
arising from \emph{group extensions} and \emph{corners of $($cocycle$)$ actions},
the resulting maps are typically discontinuous.
However we note that, when $\alpha$ is a homomorphism, it is automatically continuous,
because $\Aut(A)$ is a Polish group by \cite{Wil}, Theorem D.11.
Similarly, for ordinary actions $\alpha \colon G \acts A$,
any Borel $1$-cocycle $\fu$ of $\alpha$ is automatically strictly continuous.
Indeed the map $G\ni s\mapsto \fu_s s\in\cM( A \rca{\alpha} G)^{\rm u}$ defines a Borel homomorphism by the cocycle relation,
and hence it must be continuous in the strict topology again by \cite{Wil}, Theorem D.11.
This proves the strict continuity of $\fu$.
As a result, amenability of \Cs-dynamical systems (see \cite{AD}, \cite{OS})
is preserved by exterior equivalence.

\begin{Rem}\label{Rem:ambient}
We note that any cocycle action $(\alpha, \fw) \colon G \acts A$ of an lcsc group
admits an ambient cocycle action whose $2$-cocycle is a coboundary.
To see this, we define $C$ to be the \Cs-subalgebra of $\cM(A \rca{\alpha, \fw} G)$ generated by $A$ and $A\rca{\alpha, \fw} G$.
Then $\cM(C) \subset \cM(A\rca{\alpha, \fw} G)$ contains
$s^\fw$ and $\fw_{s, t}$ for $s, t\in G$.
Thus we have the conjugate (Borel) cocycle action $(\ad(\bullet^\fw), \fw)\colon G \acts C$.
One has $A \subset (C, \ad(\bullet^\fw), \fw)$,
and $(\ad(\bullet^\fw), \fw)\colon G \acts C$ is exterior equivalent to the trivial action via the (strict) Borel map $\Gamma \ni s \mapsto s^\fw\in \cM(C)^{\rm u}$.
\end{Rem}
\subsection{Corner cocycle actions}
Here for our purpose, we review the corner cocycle actions.
Because the author was not able to find an appropriate reference, we record the details of the construction for the reader's convenience.

For an lcsc group $G$, let $(\alpha, \fw) \colon G \acts A$ be a cocycle action on a separable \Cs-algebra $A$.
Assume that $p\in A^{\rm p}$ satisfies the following condition:
the projections $\alpha_g(p)$; $g\in G$ are mutually Murray--von Neumann equivalent.
Then one can define the \emph{corner cocycle action} $(\alpha^p, \fw^p) \colon G \acts pAp$ as follows.
We first observe the following lemma.
\begin{Lem}\label{Lem:Borel}
Keep the above setting. 
Then there is a Borel map $\fv \colon G \rightarrow A$ satisfying
\[\fv_g\fv_g^\ast = p,\quad \fv_g^\ast \fv_g=\alpha_g(p)\]
for all $g\in G$.
\end{Lem}
\begin{proof}
For each $g\in G$, set
\[U_g:=\{s\in G: \alpha_s(p)\approx_{1}\alpha_g(p)\}.\]
Note that each $U_g$ is Borel in $G$.
Since $A$ is separable, one can find a sequence $(g_n)_{n\in \IN}$ in $G$ with
\[\bigcup_{n\in \IN} U_{g_n}=G.\]
For each $n\in \IN$,
we get a Borel map $\fv^{(n)}\colon U_{g_n} \rightarrow A$ with
\[\fv^{(n)}_s (\fv^{(n)}_s)^\ast=\alpha_{g_n}(p),\quad (\fv^{(n)}_s)^\ast\fv^{(n)}_s= \alpha_{s}(p)\]
by defining $\fv^{(n)}_s$ to be the partial isometry part of the polar decomposition of $\alpha_{g_n}(p)\alpha_s(p)$ for $s \in U_{g_n}$ (which exists in $A$ by the norm condition).
For each $n\in \IN$, we pick $v_n \in A$ with
\[v_n v_n^\ast=p,\quad v_n^\ast v_n= \alpha_{g_n}(p).\]
(This is the point where we use the assumption on $p$.)
Now the desired Borel map $\fv\colon G \rightarrow A$ is given by
\[\fv_g:=v_n \fv^{(n)}_g\quad {\rm~ for~} g\in U_{g_n} \setminus \bigcup_{k=1}^{n-1} U_{g_k},~ n\in \IN.\]
\end{proof}
Now we recall the definition of the corner cocycle action. 
Keep the above setting.
We choose a Borel map $\fv\colon G \rightarrow A$ as in Lemma \ref{Lem:Borel}.
Then define the maps
\[\alpha^p \colon G \rightarrow {\rm Aut}(pAp) \quad{\rm and}\quad \fw^p \colon G \times G \rightarrow (pAp)^{\rm u}\]
 to be
\[\alpha^p_s(a):=\fv_s \alpha_s(a) \fv_s^\ast \quad {\rm~for~}a\in pAp, s\in G,\]
\[\fw^p_{s, t}:=\fv_s \alpha_s(\fv_t)\fw_{s, t} \fv_{st}^\ast \quad {\rm~for~}s, t\in G.\]
Since $\fw_{s, t} \alpha_{st}(p)=\alpha_s(\alpha_t(p)) \fw_{s, t}$ for $s, t\in G$,
the values of $\fw^p$ are indeed contained in $(pAp)^{\rm u}$.
The pair $(\alpha^p, \fw^p)$ defines a cocycle action of $G$ on $pAp$.

We remark that the definition of $(\alpha^p, \fw^p)$ depends on the choice of $\fv$, but is independent up to exterior equivalence.
Indeed for any other Borel map $\fv'$ with the above properties,
the Borel map $\fu\colon G \rightarrow (pAp)^{\rm u}$ given by
$\fu_g:=\fv'_g \fv_g^\ast$ witnesses the exterior equivalence of the two resulting cocycle actions.
As this difference is not important in this article, we omit the label $\fv$ in the notation $(\alpha^p, \fw^p)$.

We first record the following basic result.
\begin{Lem}\label{Lem:cvss}
Let $A$ be a unital separable \Cs-algebra.
Let $(\alpha, \fw)\colon G \acts A \otimes \IK$ be a cocycle action.
We fix a matrix unit $(e_{i, j})_{i, j\in \IN}$ of $\IK$.
Assume that, with $p:=1\otimes e_{1, 1}$,
the projections $\alpha_g(p); g\in G$ are mutually Murray--von Neumann equivalent.
We identify $A$ with $A\otimes \IC e_{1, 1}$ in the obvious way.
Then $(\alpha^p, \fw^p)\otimes 1_{\IK}$ is exterior equivalent to $(\alpha, \fw)$.
\end{Lem}
\begin{proof}
We fix a Borel map $\fv$ as in Lemma \ref{Lem:Borel} for $p$,
and we define $(\alpha^p, \fw^p)$ via $\fv$.
Define the Borel map $\fu \colon G \rightarrow \cM(A\otimes \IK)^{\rm u}$ by
\[\fu_g:=\sum_{n\in \IN}\alpha_g(e_{n, 1})\fv_g^\ast e_{1, n}.\]
Here and below the series are taken in the strict topology. 
Note that 
\[\fu_g e_{i, 1}=\alpha_g(e_{i, 1})\fv_g^\ast\quad {\rm~ for~} g\in G,~ i\in \IN.\]

Direct calculations show that
\begin{align*}
\fu_g(\alpha^p_g \otimes \id_{\IK})(a \otimes e_{i, j})\fu_g^\ast&=\fu_g e_{i, 1}\alpha_g^p(a\otimes e_{1, 1}) e_{1, j} \fu_g^\ast\\
&= \fu_g e_{i, 1} \fv_g\alpha_g(a\otimes e_{1, 1}) \fv_g^\ast e_{1, j}\fu_g^\ast \\
&=\alpha_g(a\otimes e_{i, j}),\\
\fu_g(\alpha^p_g \otimes \id_{\IK})(\fu_h)(\fw^p\otimes 1_{\IK})_{g, h} \fu_{gh}^\ast&= \sum_{n\in \IN} \alpha_g(\fu_h)\fu_g e_{n, 1} \fv_g \alpha_g(\fv_h)\fw_{g, h} \fv_{gh}^\ast e_{1, n} \fu_{gh}^\ast\\
&= \sum_{n\in \IN} (\alpha_g\circ\alpha_h)(e_{n, 1})\fw_{g, h} \alpha_{gh}(e_{1, n})\\
&=\fw_{g, h}
\end{align*}
for $a\in A, i, j\in \IN, g, h\in G$.
Thus $\fu$ witnesses the desired equivalence.
\end{proof}

Now we prove the main theorem of this section.
This is an application of the recent developments on amenable actions on simple \Cs-algebras \cite{Suzeq}, \cite{OS}
and the existence of non-exact groups \cite{Gro}, \cite{Osa}.
\begin{Thm}\label{Thm:untwist}
Let $G$ be a non-exact lcsc group. Then there is a cocycle action
$(\alpha, \fw) \colon G \acts \mathcal{O}_2$ with the following property.
For any cocycle action $(\beta, \fx) \colon G \acts A$ on a unital separable \Cs-algebra,
the $(\alpha\otimes \beta)$-$2$-cocycle $\fw\otimes \fx$ is not a coboundary.
\end{Thm}
\begin{proof}
By applying Theorem 6.1 of \cite{OS} to the cone $\id_{C_0(0, 1]} \otimes {\rm L}$ of the left translation action ${\rm L}\colon G \acts C_0(G)$,
we obtain an amenable action $\sigma \colon G \acts \mathcal{O}_2 \otimes \IK$.
(Here, certainly, we use the classification theorem of Kirchberg--Phillips \cite{Kir}, \cite{Phi}.)
Note that any two nonzero projections in $\mathcal{O}_2\otimes \IK$ are Murray--von Neumann equivalent by \cite{Cun}.
Therefore one can define the corner cocycle action
$(\alpha, \fw)\colon G \acts \mathcal{O}_2$ of $\sigma$ at $1\otimes e$,
where $e$ is any minimal projection in $\IK$.
We will show that $(\alpha, \fw)$ possesses the desired property.

Let $(\beta, \fx) \colon G \acts A$ be any cocycle action on a unital separable \Cs-algebra $A$.
Suppose that $(\alpha\otimes \beta, \fw \otimes \fx)$ is exterior equivalent to a genuine action $\gamma$.
Then $(\alpha\otimes \beta, \fw \otimes \fx)\otimes 1_{\IK}$ is exterior equivalent to $\gamma \otimes 1_{\IK}$.
By Lemma \ref{Lem:cvss}, $(\alpha, \fw)\otimes 1_{\IK}$ is exterior equivalent to $\sigma$.
Also $(\beta, \fx)\otimes 1_{\IK}$ is exterior equivalent to a genuine action by the stabilization trick \cite{Sut}, \cite{PR}.
Since $\IK\cong \IK \otimes \IK$, these observations show that $\gamma \otimes 1_{\IK}$ is exterior equivalent to an amenable action.
This implies the amenability of $\gamma$.
Since $G$ is non-exact, this contradicts to Corollary 3.6 of \cite{OS}.
\end{proof}

\begin{Rem}We remark that both the non-exactness and the unital assumptions of Theorem \ref{Thm:untwist} are essential.
Indeed:
\begin{enumerate}
\item When $G$ is exact, such a cocycle action does not exist.
Indeed by combining the stabilization trick \cite{Sut}, \cite{PR} and the Gabe--Szabo classification theorem \cite{GS2},
one concludes the following statement:
Let $\beta \colon G \acts \mathcal{O}_2$ be an amenable, isometrically shift-absorbing action \cite{GS2}.
Then, for any cocycle action $(\alpha, \fw) \colon G \acts A$ on a simple separable nuclear \Cs-algebra
(in particular when $A=\mathcal{O}_2$),
the $(\alpha\otimes \beta)$-$2$-cocycle $\fw\otimes 1$ is a coboundary.
This follows from the proof of Theorem 5.6 of \cite{GS2}.
The proof is similar to the proof of Theorem 6.10 in \cite{GS2},
but for the reader's convenience, we give a sketch of the proof.

We first note that when $A$ is non-unital, the claim follows from the stabilization trick (see also Remark 1.6 of \cite{GS2}).
To show the claim in the unital case,
we choose an action $\gamma\colon G \acts A\otimes \mathcal{O}_2\otimes \IK $ which is exterior equivalent to the stabilization $(\alpha\otimes \beta, \fw\otimes 1)\otimes 1_{\IK}$ (which exists by the stabilization trick).
Note that for a minimal projection $e\in \IK$,
 the corner cocycle action of $\gamma$ at $p:=1_A \otimes 1_{\mathcal{O}_2}\otimes e$ is exterior equivalent to $(\alpha\otimes \beta, \fw\otimes 1)$.
Let $\eta \colon G \acts \IC$ denote the trivial action.
Then by applying the proof of Theorem 5.6 to $0\in {\rm KK}^G(\eta, \gamma)$,
we obtain a proper cocycle embedding
$(\varphi, \fv)\colon (\IC, \eta) \rightarrow (A\otimes \mathcal{O}_2\otimes \IK, \gamma)$
which represents $0$ in ${\rm KK}^G(\eta, \gamma)$.
This implies that the corner cocycle action of $\gamma$ at $q:=\varphi(1)$
is exterior equivalent to a genuine action.
Since $q$ is Murray--von Neumann equivalent to $p$,
this proves the claim.

\item The conclusion of Theorem \ref{Thm:untwist} fails when $A$ is non-unital,
by the stabilization trick \cite{Sut}, \cite{PR}.\end{enumerate}
\end{Rem}

\begin{Rem}Here we give a few comments on von Neumann algebras.
\begin{enumerate}
\item
Sutherland \cite{Sut} proved that any $2$-cocycle of a cocycle action on a properly infinite von Neumann algebra is a coboundary.
Thus Theorem \ref{Thm:untwist} is a very unique phenomenon in \Cs-algebra theory.
\item
Still it would be an interesting and deep question if an analogous example exists
in the realm of type II$_1$ factors.
We note that cocycle actions which are shown by Popa in \cite{Pop}, Theorem 3.2 to be unable to untwist
readily become coboundaries after tensoring with any type II$_1$ factor (equipped with the trivial action), because they are corner cocycle actions (of Bernoulli shift actions).
\end{enumerate}
\end{Rem}

\section{K$_0$-obstruction: New tensor component $\mathcal{O}_2$ is necessary and minimal}\label{section:O2}
In this section, we show that
the operation appearing in Theorem \ref{Thmint:Cs},
namely adding the new tensor component $\mathcal{O}_2$, is in general necessary to obtain the clean splitting.

We first see that the necessity of a new tensor component already appears in the following elementary and well-known example.
Denote by $\mathbb{M}_{2^\infty}$ the CAR algebra.
We identify $\mathbb{M}_{2^\infty}$ with the gauge fixed point algebra of $\mathcal{O}_2$ in the standard way.
Then, as observed in \cite{RorIrr}, the inclusion $\mathbb{M}_{2^\infty} \subset \mathcal{O}_2$ is \Cs-irreducible.
Let $\alpha \colon \IZ \acts \mathcal{O}_2$ be an automorphism commuting with the gauge action.
(For instance, any quasi-free automorphism commutes with the gauge action.)
Then for any $\chi\in \widehat{\IT} \setminus \{1\}$, the set
\[C:=\left\{x\in \mathcal{O}_2 \rca{\alpha} \IZ: \E_n(x) \in (\mathcal{O}_2)_{\chi^n} {\rm ~for~all~}n\in \IZ \right\}\]
sits in $\Int{\mathbb{M}_{2^\infty}}{\mathcal{O}_2\rca{\alpha} \IZ}$.
(Indeed the subspace family $((\mathcal{O}_2)_{\chi^n})_{n\in \IZ}$ satisfies the axioms in Remark \ref{Rem:partial}.)
We observe that $\E_n(C) \cap (\mathcal{O}_2)^{\rm u}=\emptyset$ for $n\in\IZ\setminus\{0\}$.
To prove this, since $\widehat{\IT}$ is torsion-free,
it suffices to show the next claim: For every $1\neq \kappa \in \widehat{\IT}$, $(\mathcal{O}_2)_\kappa \cap ( \mathcal{O}_2)^{\rm u}=\emptyset$.
To show this claim, choose $n\in \IZ \setminus \{0\}$ satisfying $\kappa(z)=z^n$.
By considering the adjoint $(\mathcal{O}_2)_\kappa^\ast=(\mathcal{O}_2)_{\kappa^{-1}}$ instead of $(\mathcal{O}_2)_\kappa$
if necessary, we may assume $n\geq 1$.
Then standard calculations in $\mathcal{O}_2$ show that
\[(\mathcal{O}_2)_\kappa=\spa \{s_\mu\cdot \mathbb{M}_{2^\infty}: \mu\in \{1, 2\}^n\}.\]
Here and below,
we write $s_1$, $s_2$ for the generators of $\mathcal{O}_2$,
and then as usual, for $\mu =(\mu_1, \mu_2, \ldots, \mu_n)\in \{1, 2\}^n$,
we put $s_\mu:= s_{\mu_1}s_{\mu_2} \cdots s_{\mu_n}$.
Suppose that $(\mathcal{O}_2)_\kappa$ contains a unitary element, say
\[u=\sum_{\mu\in \{1, 2\}^n} s_\mu x_\mu; \quad x_\mu \in \mathbb{M}_{2^\infty}.\]
Then, from the condition $u^\ast u=1$, one has
\[\sum_{\mu\in \{1, 2\}^n} x^\ast_\mu x_\mu=1.\]
The other condition $u u^\ast=1$ yields
\[x_\mu x_\nu^\ast =s_\mu^\ast uu^\ast s_\nu = \delta_{\mu, \nu} 1 \quad {\rm ~for~}\mu, \nu \in \{1, 2\}^n.\]
This contradicts to the finiteness of $\IM_{2^\infty}$.
Thus $(\mathcal{O}_2)_\kappa \cap ( \mathcal{O}_2)^{\rm u}=\emptyset$.
Consequently $C$ does not (naturally) split into $D\rca{{}^\fu \alpha, {\rm d}\fu} \Lambda$ by any
$(D, \Lambda)$-respecting map
\[\fu \colon \Lambda \rightarrow (\mathcal{O}_2)^{\rm u},\quad 
D\in {\rm Int}(\mathbb{M}_{2^\infty} \subset \mathcal{O}_2),\quad \Lambda<\IZ.\]
Indeed if $C=D\rca{{}^\fu \alpha, {\rm d}\fu} \Lambda$,
then one must have $\Lambda=\IZ$,
and $\fu_n\in \E_n(C)=(\mathcal{O}_2)_{\chi^n}$ for $n\in \IZ$, which is a contradiction.

We point out that, by modifying the above construction, one can also give counter-examples
for the other cyclic groups $C_k:=\IZ/k\IZ$; $k=2, 3, \ldots$,
and hence for all non-trivial discrete groups.
To see this, choose any action $\alpha\colon C_k \acts \mathcal{O}_2$ commuting with the gauge action.
Let $\chi$ denote the character on $\IT$ given by $\chi_z=z$ for $z\in \IT$.
For $\bar{n}:=n+k\IZ \in C_k$, set
\[X_{\bar{n}}:=\overline{\rm span}\{(\mathcal{O}_2)_{\chi^{l}}:l\in \bar{n}\}.\]
Then $B:=X_{\bar{0}}$ is a \Cs-subalgebra of $\mathcal{O}_2$ generated by $\{s_\mu:\mu\in \{1, 2\}^k\}$.
Hence $B$ is isomorphic to the Cuntz algebra $\mathcal{O}_{2^k}$.
We show that
\[C:={\rm span}\{X_{\bar{n}} \bar{n}:\bar{n}\in C_k\} \in \Int{\IM_{2^\infty}}{\mathcal{O}_2\rca{\alpha} C_k}\]
gives a counter-example for the clean splitting.
As before, for this, it suffices to show that $X_{\bar{1}} \cap (\mathcal{O}_2)^{\rm u}=\emptyset$.
To see this, note first that
\[X_{\bar{1}}=(s_1 s_1^\ast + s_2 s_2^\ast)X_{\bar{1}}=s_1B+s_2 B.\]
Suppose that $X_{\bar{1}}$ contains a unitary element, say
\[u=s_1b_1 +s_2 b_2; \quad b_1, b_2\in B.\]
Then, similar to the previous case, one has
\[b_1^\ast b_1 + b_2^\ast b_2=1, \quad b_i b_j ^\ast=\delta_{i, j} 1 \quad {\rm for}~ i, j\in \{1, 2\}.\]
Thus $\{b_1, b_2\}$ generates a unital \Cs-subalgebra of $B$ isomorphic to $\mathcal{O}_2$.
This is a contradiction, because $[1_B]_0 \neq 0$ in K$_0(B)$ by \cite{Cun}.
Thus $C$ fails the clean splitting.

Next, by modifying the first construction,
we also give examples which show that our adding tensor component $\mathcal{O}_2$ cannot be replaced by any other unital \Cs-algebra $A$ with $[1_A]_0 \neq 0$ in ${\rm K}_0(A)$. 
Before the construction, we state its consequence.
Since any unital properly infinite \Cs-algebra $A$ with $[1_A]_0=0$
 contains $\mathcal{O}_2$ as a unital \Cs-algebra (by \cite{Cun}),
 this shows that our choice of the tensor component $\mathcal{O}_2$ is a minimal possible choice among \emph{regular} unital simple \Cs-algebras.
We also note that, even considering pathological unital simple \Cs-algebras (see \cite{Ror03}),
the condition $[1_B]_0=0$ for a unital simple \Cs-algebra $B$ forces
the matrix amplification $\IM_k(B)$ for a sufficiently large $k\in \IN$ to contain $\mathcal{O}_2$ as a unital \Cs-algebra.
Thus there is no essential room of an improvement for the choice of $\mathcal{O}_2$.

Let us now construct the desired action.
Choose a transcendental number $x \in (0, 1)$.
Consider the ordered subgroup $G:=\mathbb{Z}[x, x^{-1}] < \mathbb{R}$.
Since $G$ is dense in $\IR$, it is a simple dimension group as proved by Elliott (cf.~the Effros--Handelman--Shen theorem \cite{EHS}).
Fix a unital (simple) AF-algebra $B$ with
\[({\rm K}_0(B), {\rm K}_0(B)_+, [1_B]_0) \cong (G, G_+, 1).\]
We identify K$_0(B)$ with $G$ via this isomorphism.
Take $p\in B^{\rm p}$ with $[p]_0=x$.
Clearly the map $y\mapsto xy$ defines an order automorphism of $G$.
Hence by the Bratteli--Elliott classification theorem, one can choose an isomorphism
$\rho \colon B \rightarrow pBp$. 
Let $C:=B \rca{\rho} \IN$ be the crossed product of the corner endomorphism $\rho$ (see \cite{Pas} for its definition and basic properties).
We note that the inclusion $B \subset C$ is \Cs-irreducible.
To see this, recall that by Theorem 2 of \cite{Pas},
the stabilized inclusion $\IK \otimes B \subset \IK \otimes C$ is isomorphic to $\IK\otimes B \subset (\IK\otimes B)\rca{\tilde{\rho}} \IZ$,
where $\tilde{\rho} \in {\rm Aut}(\IK \otimes B)$ is an automorphism extension of $\rho$.
Since $\rho_{\ast, 0} \in {\rm Aut}({\rm K}_0(B))$ has infinite order,
 the $\IZ$-action $\tilde{\rho}$ is pointwise outer.
By the proof of Kishimoto's simplicity theorem \cite{Kis},
the inclusion $\IK\otimes B \subset (\IK\otimes B)\rca{\tilde{\rho}} \IZ$, and hence $B\subset C$,
is \Cs-irreducible.

Let $s\in C$ denote the canonical implementing isometry element of $\rho$. Note that $ss^\ast =p$.
Let $\sigma \colon \IT \acts C$ be the gauge action.
(Namely, the action given by $\sigma_z(b)=b$ and $\sigma_z(s)=zs$ for $b\in B$ and $z\in \IT$.)
We recall that $A$ is a unital \Cs-algebra with $[1_A]_0 \neq 0$.
On $C$ and $A \otimes C$, we equip with the $\IT$-actions $\sigma$ and $\id_A\otimes \sigma$ respectively.
Let $\chi$ be the character on $\IT$ given by $\chi_z=z$.
Since one has the completely contractive projection $\E^\IT_\chi \colon C \rightarrow C_\chi$,
it follows from the description of the algebraic elements in the formula (*) in page 114 of \cite{Pas} that
$C_\chi=Bs$.
This yields $(A\otimes C)_\chi=(A \otimes B) \cdot (1 \otimes s)$.
We show $(A\otimes C)_\chi \cap(A\otimes C)^{\rm u} = \emptyset$.
Suppose that one has a unitary element $u$ of the form
\[u=x(1\otimes s); \quad x\in A \otimes B.\]
Then from the condition $u u^\ast =1$, one has $x (1\otimes p) x ^\ast=1$.
Also, from $u^\ast u=1$, one obtains $(1\otimes s)^\ast x^\ast x (1\otimes s)=1$.
These equations show that $v:=x (1\otimes p)\in A\otimes B$ satisfies
$v^\ast v =1 \otimes p$ and $vv^\ast =1$.
This shows that $[1\otimes p]_0=[1]_0$ in K$_0(A\otimes B)$.
At the same time, since $B$ is an AF-algebra and ${\rm K}_0(B)$ has the $\mathbb{Z}$-basis $(x^n)_{n\in \IZ}$,
by the continuity of the K$_0$-groups,
 one has a group isomorphism
\[{\rm K}_0(A\otimes B) \cong {\rm K}_0(A) \mathop{\otimes}\limits_{\mathbb{Z}} \left(\bigoplus_{\IZ} \IZ\right) \]
which sends $[1_A \otimes p]_0$ to $[1_A]_0 \otimes \delta_1$ and $[1]_0$ to $[1_A]_0 \otimes \delta_0$.
Since $[1_A]_0\neq 0$, these two elements are different.
This is a contradiction.
Hence $(A\otimes C)_\chi \cap (A\otimes C)^{\rm u}=\emptyset$.

Thus for any action $\gamma \colon \IZ \acts C$ commuting with $\sigma$
and for any action $\alpha \colon \IZ \acts A$,
\begin{align*}
D:=&\{x\in (A\otimes C)\rca{\alpha\otimes \gamma} \IZ: \E_n(x)\in (A\otimes C)_{\chi^n}{\rm~ for~}n\in \IZ\}\\
&\in \Int{A\otimes B}{(A\otimes C)\rca{\alpha\otimes \gamma} \IZ}
\end{align*}
fails the clean splitting by the same reason as the previous example.

\section{Von Neumann algebra case}\label{section:vn}
In this section, we establish the von Neumann algebra version of our splitting theorem.
Throughout this section, to avoid inessential complexity, without stated, our von Neumann algebras (except ultraproducts and their relatives) are assumed to have a separable predual.

For von Neumann algebras, inspired by the studies of group actions (see \cite{Con}, \cite{Jon80}, \cite{Oc}, \cite{KST} etc.) and subfactor theory (see \cite{Kaw}, \cite{Pop95}, etc.),
as in \cite{SuzCMP},
we formulate central freeness of $\Gamma$-von Neumann subalgebras in terms of the relative central sequence algebras.
For this purpose, throughout this section, we fix a free ultrafilter $\omega$ on $\IN$.
We refer the reader to Section 4 (between Remark 4.2 and Definition 4.3) of \cite{SuzCMP} for the definition of the relative central sequence algebras $N^\omega \cap M_\omega$,
and to \cite{AH} for basic facts on ultraproducts of von Neumann algebras.
\begin{Def}\label{Def:cfvn}
Let $N\subset M$ be a von Neumann algebra inclusion.
We say that $\alpha \in {\rm Aut}(M)$ with $\alpha(N)=N$
is \emph{centrally free for} $N\subset M$
if there is a partition of unity $(p_n)_{n\in \IN}$ in $(N^\omega \cap M_\omega)^{\rm p}$
with $p_n\alpha(p_n)=0$ for $n\in \IN$.
We say that an inclusion $N \subset (M, \alpha, \fw)$ of $($twisted$)$ $\Gamma$-von Neumann algebras is \emph{centrally $\Gamma$-free}
if every $\alpha_s$; $s\in \Gamma \setminus \{e\}$, is centrally free for $N\subset M$.
\end{Def}
By Connes's non-commutative Rohlin lemma (see Theorem 1.6 and Lemma 2.6 in Chapter XVII of \cite{Tak3}), the automorphism $\alpha \in {\rm Aut}(M)$ with $\alpha(N)=N$
is centrally free for $N\subset M$ if and only if
the induced automorphism on $N^\omega \cap M_\omega$ is properly outer.
In particular, when $N=M$, the definition coincides with Ocneanu's central freeness for automorphisms on von Neumann algebras \cite{Oc}.

We use the following version of Lemma 4.7 of \cite{SuzCMP}.
This is a consequence of the standard reindexation argument.
\begin{Lem}\label{Lem:reind}
Let $N\subset (M, \alpha, \fw)$ be a centrally $\Gamma$-free von Neumann algebra inclusion.
Then for any finite subset $F\subset \Gamma \setminus \{e\}$,
there is a partition of unity $(p_n)_{n\in \IN}$ in $(N^\omega \cap M_\omega)^{\rm p}$ with
$p_n \alpha_s(p_n)=0$ for $n\in \IN$ and $s\in F$.
\end{Lem}
For a faithful normal state $\varphi$ on a von Neumann algebra $M$, we set
\[\|x\|_\varphi:=\sqrt{\varphi(x^\ast x)}\]
for $x\in M$. This defines a norm on $M$.
On $(M)_1$, the topology of
$\|\bullet\|_\varphi$ coincides with the $\sigma$-strong operator topology.

The next lemma is the von Neumann algebra version of Lemma \ref{Lem:cf}.
\begin{Lem}\label{Lem:cfv}
Let $N\subset (M, \alpha, \fw)$ be a centrally $\Gamma$-free inclusion of $\Gamma$-von Neumann algebras.
Then for any $s\in \Gamma$ and any $C\in \Int{N}{M\vnc{\alpha, \fw} \Gamma}$,
one has $\E_s(C)s^\fw \subset C$.
\end{Lem}
\begin{proof}
We fix a faithful normal state $\varphi$ on $M$.
Put $\psi:=\varphi \circ \E$, which is a faithful normal state on $M\vnc{\alpha, \fw} \Gamma$.
Let $s\in \Gamma$, $C\in \Int{N}{M\vnc{\alpha, \fw} \Gamma}$, and $c\in (C)_1$ be given.
It suffices to show the following claim:
For any $\epsilon>0$, there is $d\in (C)_1$ with
$\|\E_s (c)s^\fw - d\|_{\psi}<2\epsilon$.
We prove this claim.

We apply the Kaplansky density theorem to pick $c_0\in M \rtimes_{{\rm alg}, \alpha, \fw} \Gamma$ with
$\|c-c_0\|_{\psi}<\epsilon$, $\|c_0\|\leq 1$.
Write
\[c_0=\sum_{t\in F} a_t t^\fw s^\fw,\]
where $F=\{e, s_1, s_2, \ldots, s_k\}\subset \Gamma$ with $|F|=k+1$, $a_t\in M$ for $t\in F$, and $a_e=\E_s(c)$. 
If $k=0$, the claim is trivial.
Otherwise we apply Lemma \ref{Lem:reind} to $\{s_1, \ldots, s_k\} \subset \Gamma \setminus \{e\}$.
Then we obtain a partition of unity $(p_n)_{n\in \IN}$ in $(N^\omega \cap M_\omega)^{\rm p}$ with
$p_n \alpha_{s_i}(p_n)=0$ for $n\in \IN$ and $i=1, 2, \ldots, k$.
This implies
\[\sum_{n\in \IN} p_n a_{t}\alpha_{t}(p_n)=
 \left\{
\begin{array}{ll}
a_e & \text{when } t=e,\\
0 & \text{when } t\in \{s_1, \ldots, s_k\}.
\end{array}
\right.\]

Choose $l\in \IN$ with
$\sum_{n >l} \|p_n\|_{\varphi^\omega}<\epsilon$.
Pick a representing sequence $(p^{(n)}_m)_{m\in \IN} \subset N^{\rm p}$ of $p_n$ for each $n \leq l$.
We may arrange that the sequence $(p^{(n)}_m)_{n =1}^l$ is mutually orthogonal for each $m\in \IN$.
Then, for all $m\in \IN$ sufficiently close to $\omega$, we have
\[\| \sum_{n=1}^ lp^{(n)}_m c_0 \alpha_s^{-1}(p^{(n)}_m) - a_e s^{\fw}\|_{\psi}<\epsilon.\]

Since $p_n\in M_\omega$ for $n\in \IN$, one has
\begin{align*}
\lim_{m\rightarrow \omega} \|\sum_{n=1}^l p^{(n)}_m (c_0-c)\alpha_s^{-1}(p^{(n)}_m)\|_{\psi}^2 
&\leq \lim_{m\rightarrow \omega}\sum_{n=1}^l  \psi(\alpha_s^{-1}(p^{(n)}_m)(c_0-c)^\ast(c_0-c)\alpha_s^{-1}(p^{(n)}_m))\\
&= \lim_{m\rightarrow \omega}\sum_{n=1}^l  \psi(\alpha_s^{-1}(p^{(n)}_m)(c_0-c)^\ast(c_0-c))\\
 &\leq\|c_0-c\|_{\psi}^2<\epsilon^2.
\end{align*}
Here for the last inequality, we use the following fact:
For any $(q_n)_{n\in \IN} \in \ell^\infty(\IN, N)^{\rm p}$ representing an element of $N^\omega \cap M_\omega$ and any $a\in M\vnc{\alpha, \fw} \Gamma$,
one has
\[\lim_{n \rightarrow \omega} \psi(q_n a) = \lim_{n\rightarrow \omega} \psi(q_n a q_n) \geq 0.\]
Therefore, for all $m\in \IN$ sufficiently close to $\omega$, we obtain
\[\| \sum_{n=1}^lp^{(n)}_m (c_0-c) \alpha_s^{-1}(p^{(n)}_m) \|_{\psi}<\epsilon.\]

Thus, for a suitable $m\in \IN$,
the element
\[d:=\sum_{n=1}^l p^{(n)}_m c \alpha_s^{-1}(p^{(n)}_m)\in (C)_1\]
satisfies the desired condition $\|d-\E_s(c)s^\fw\|_\psi<2\epsilon$.
\end{proof}

The next two lemmas are von Neumann algebra versions of Lemma \ref{Lem:unitary}.
Since the K$_0$-group obstruction is very little in factors,
we only need mild assumptions on the Murray--von Neumann type instead of the $\mathcal{O}_2$-absorption.
(However the statement fails in general. See Remark \ref{Rem:countexv}.)
\begin{Lem}\label{Lem:unitaryv}
Let $N\subset (M, \alpha, \fw)$ be a centrally $\Gamma$-free irreducible $\Gamma$-subfactor.
Assume that $N$ is of infinite type.
Then for any $s\in \Gamma$ and any $C\in \Int{N}{M\vnc{\alpha, \fw} \Gamma}$,
one has either $\E_s(C)=\0$ or $\E_s(C)\cap M^{\rm u}\neq \emptyset$.
\end{Lem}
\begin{proof}
Pick a sequence $(v_n)_{n\in \IN}$ of partial isometry elements in $N$ with 
\[p:=v_1 v_1^\ast =v_n ^\ast v_n \quad {\rm~for~all~}n\in \IN, \quad \sum_{n\in \IN} v_n v_n^\ast =1.\]
For a weak-$\ast$ closed $(N, N)$-subbimodule $X \subset M$, we put $X_p:=p X p$.
Define an isomorphism $\Theta \colon M \rightarrow \IB \btimes M_p$ to be
\[\Theta(x):=[v_n^\ast x v_m]_{n, m\in \IN}.\]
Then $\Theta(X)=\IB\btimes X_p$ (the weak-$\ast$ closure of the algebraic tensor product $\IB\odot X_p$ in $\IB\btimes M_p$) for all such subbimodules $X \subset M$.

Now assume that $\E_s(C) \neq \0$ for a given $s\in \Gamma$.
Note that $\E_s(C)s^\fw\subset C$ by Lemma \ref{Lem:cfv}.
Hence $\E_s(C)$ is a weak-$\ast$ closed $(N, N)$-subbimodule of $M$.
Hence one can find a nonzero element $d\in M_p$ with $x:=1_{\IB} \otimes d \in \Theta(\E_s(C))$.
Since $xx^\ast \in \Theta(C \cap M)$,
we have $dd^\ast \in (C \cap M)_p$.
Consider the Borel function
\[f(t):=t^{-1}\chi_{[\|d\|/2, \infty)}(t)\]
on $[0, \infty)$.
Set $y:=1\otimes f(|d^\ast|)d \in \Theta(\E_s(C))$.
Then $y$ is a nonzero partial isometry element.
Hence $yy^\ast$ and $y^\ast y$ are nonzero projections
in $1_{\IB}\otimes (C \cap M)_p$ and in $1_{\IB}\otimes (\alpha_s(C \cap M))_p$ respectively.
Since $(C\cap M)_p$ is a factor, by the Comparison Theorem, one can find isometry elements $v \in \IB\btimes (C \cap M)_p$ and $w\in \IB\btimes (\alpha_s(C \cap M))_p$
with $vv^\ast = yy^\ast$, $ww^\ast = y^\ast y$.
These relations yield $\Theta^{-1}(v^\ast y w) \in \E_s(C)\cap M^{\rm u}$.
\end{proof}

\begin{Lem}\label{Lem:type}
Let $N \subset (M, \alpha, \fw)$ be an irreducible, centrally $\Gamma$-free $\Gamma$-subfactor.
Assume that $M$ is of finite type.
Then for any $C\in \Int{N}{M \vnc{\alpha, \fw} \Gamma}$ and any $s\in \Gamma$,
one has
either $\E_s(C)=\0$ or $\E_s(C)\cap M^{\rm u}\neq \emptyset$.
\end{Lem}
\begin{proof}
Let $s\in \Gamma$ be an element with $\E_s(C) \neq \0$.
We define $\cV$ to be the set of all partial isometry elements contained in $\E_s(C)$.
Note that $0\in \cV$.
On $\cV$, we equip with the partial order $\preceq$ by declaring $v \preceq w$
when $v=vv^\ast w$.

Since $\E_s(C)s^\fw=(Ms^\fw) \cap C$ by Lemma \ref{Lem:cfv}, $\E_s(C)$ is weak-$\ast$ closed in $M$.
Hence, by Zorn's lemma,
one has a maximal element $v$ in $(\cV, \preceq)$.
We will show that $v \in M^{\rm u}$.

Suppose that $v\not\in M^{\rm u}$.
Since $M$ is finite, neither of the projections
\[p:=vv^\ast, \quad q:=v^\ast v\]
are equal to $1_M$.
Pick a nonzero element $x\in \E_s(C)$.
As the inclusion $N\subset M$ is irreducible, one has $u_1\in N^{\rm u}$ with
$(1-p)u_1 x \neq 0$.
(Indeed, with $r:=\chi_{(0, \infty)}(xx^\ast)$ (the range projection of $x$), one has
\[0\neq \bigvee_{ u\in N^{\rm u}}uru^\ast\in N'\cap M=\IC 1_M.\]
Hence one has $u_1\in N^{\rm u}$ with $(1-p)u_1 xx^\ast u_1^\ast \neq 0$.
This gives the desired $u_1$.)
By applying the same argument once more, we get $u_2 \in N^{\rm u}$ with
$(1-p)u_1 x u_2 (1-q) \neq 0$. Since $N\subset C$, we have
$u_1 x u_2 \in \E_s(C)$.

Note that as $vs^\fw \in C$, we have $p, \alpha_s^{-1}(q) \in C \cap M$.
Hence
\[0 \neq a:=(1-p)u_1 x u_2 (1-q) \in \E_s(C).\]
Now let $a=|a^\ast| w$ be the right polar decomposition of $a$.
Since $|a^\ast| \in C \cap M$ and $\E_s(C)$ is a weak-$\ast$ closed left $(C\cap M)$-submodule of $M$, one has $w \in \cV$.
Clearly $v^\ast w=w^\ast v=0$, which shows that $v+w \in \cV$ and $v \prec v+ w$.
This contradicts to the maximality of $v$. Thus $v\in M^{\rm u}$.
\end{proof}

The next lemma is the replacement of Lemmas \ref{Lem:AP} and \ref{Lem:AP2} used in the \Cs-algebra case.
For von Neumann algebras, thanks to the von Neumann bicommutant theorem, we do not need the AP of the acting group.
The lemma also follows from the stabilization trick \cite{Sut} and Corollary 3.4 of \cite{SuzCMP}.
But here we give a direct proof, because the used reference \cite{Tak1}, Proposition V.7.14 contains non-rigorous treatments of the convergence of series expansions in crossed products.

For the definition of respecting maps for cocycle actions,
see the beginning of Section \ref{section:general}.
\begin{Lem}\label{Lem:coeffv}
Let $(\alpha, \fw)\colon \Gamma \acts M$ be a cocycle action on a von Neumann algebra.
Let $N\subset M$ be a von Neumann subalgebra, let $\Lambda<\Gamma$,
and
let $\fu\colon \Lambda \rightarrow M^{\rm u}$ be an $(N, \Lambda)$-respecting map.
Then $x\in M \vnc{\alpha, \fw} \Gamma$ is contained in
$N \vnc{{}^\fu \alpha, {}^\fu \fw} \Lambda$ if and only if it satisfies
$\E_s(x)\in N \fu_s$ for all $s\in \Lambda$ and $\E_s(x)=0$ for all $s\in \Gamma \setminus \Lambda$.
\end{Lem}
\begin{proof}
Clearly the condition is necessary.

In the rest of the proof, we identify $M\vnc{\alpha, \fw} \Gamma$ with the von Neumann algebra on a Hilbert space $\fH\otimes \ell^2(\Gamma)$
via a normal faithful regular representation.

To show the converse, take any $T\in (N \vnc{{}^\fu \alpha, {}^\fu \fw} \Lambda)'$ and $x \in M\vnc{\alpha, \fw} \Gamma$ satisfying 
$\E_s(x)\in N \fu_s$ for all $s\in \Lambda$ and $\E_s(x)=0$ for all $s\in \Gamma \setminus \Lambda$.
By the von Neumann bicommutant theorem, it suffices to show that $T$ and $x$ commute.
To prove the commutativity, we choose a net $(x_i)_{i\in I}$ in $N \rtimes_{{\rm alg}, {}^\fu \alpha, {}^\fu \fw} \Lambda$
satisfying
\[\lim_{i\in I} x_i( \xi\otimes \delta_s)=x(\xi \otimes \delta_s), \quad \lim_{i\in I} x_i^\ast(\xi \otimes \delta_s)=x^\ast(\xi \otimes \delta_s)\]
for all $\xi \in \fH$ and $s\in \Gamma$.
(For instance, define $I$ to be the set of all finite subsets of $\Lambda$ equipped with the set inclusion order.
Set
\[x_\mathfrak{F}:= \sum_{s\in \mathfrak{F}} \E_s(x) s^\fw \in N \rtimes_{{\rm alg}, {}^\fu \alpha, {}^\fu \fw} \Lambda \quad {\rm~for~}\mathfrak{F}\in I.\]
Then the net $(x_\mathfrak{F})_{\mathfrak{F}\in I}$ satisfies the desired conditions.
We remark that the net may not converge to $x$ in the strong operator topology, as it would be unbounded.)
Since each $x_i$ commutes with $T$, one has
\begin{align*}\ip{Tx(\xi \otimes \delta_s), \eta\otimes \delta_t}&=\lim_{i\in I} \ip{Tx_i (\xi\otimes\delta_s), \eta\otimes \delta_t}\\
&=\lim_{i\in I} \ip{T(\xi\otimes\delta_s), x_i^\ast(\eta\otimes \delta_t)}\\
&=\ip{T(\xi \otimes \delta_s), x^\ast(\eta\otimes \delta_t)}\\
&=\ip{xT(\xi \otimes \delta_s), \eta\otimes \delta_t} \quad\quad {\rm~for~all~} \xi, \eta\in \fH, s, t\in \Gamma. 
\end{align*}
This proves $xT=Tx$.
\end{proof}

By Lemmas \ref{Lem:unitaryv} and \ref{Lem:type}, when either $M$ is of finite type or $N$ is of infinite type,
without adding a new tensor component, we obtain the splitting result (Theorem \ref{Thmint:vn}).
\begin{proof}[Proof of Theorem \ref{Thmint:vn}]
With Lemmas \ref{Lem:cfv} to \ref{Lem:coeffv} in hand,
now the proof is essentially the same as the unital \Cs-algebra case.
In the second case (the case $N$ is infinite), one can further eliminate the $2$-cocycle,
by the stabilization trick \cite{Sut}.
\end{proof}
As a special case of Theorem \ref{Thmint:vn}, (2), we obtain the following von Neumann algebra version of Theorem \ref{Thmint:Cs}.
\begin{Cor}\label{Cor:vnsp}
Let $N \subset (M, \alpha, \fw)$ be an irreducible, centrally $\Gamma$-free $\Gamma$-subfactor.
Then for any $C\in \Int{N}{M \vnc{\alpha, \fw}\Gamma}$,
there is $\Lambda<\Gamma$ and a $(\IB \btimes (C\cap M), \Lambda)$-respecting map $\fu \colon \Lambda \rightarrow (\IB \btimes M)^{\rm u}$
satisfying ${}^\fu \fw=1$ and
\[ \IB\btimes C = (\IB\btimes (C \cap M))\vnc{{}^\fu \alpha}\Lambda.\]
\end{Cor}

\begin{Rem}\label{Rem:countexv}
Similar to the \Cs-algebra case (see Section \ref{section:O2}), in general the statement of Theorem \ref{Thmint:vn} fails.
To see this, consider the state $\varphi:=\tau \circ \E^{\IT}$ on $\mathcal{O}_2$.
Here $\tau$ denotes the unique tracial state on $\IM_{2^\infty}$.
Let $N \subset M$ be the weak-$\ast$ closures of the \Cs-algebras $\IM_{2^\infty} \subset \mathcal{O}_2$ in the GNS representation of $\varphi$.
Then $N$ is a type II$_1$ factor, $M$ is a type III factor (\cite{OP}), and the inclusion is irreducible.
As $\varphi$ is gauge invariant, the gauge action extends to the action $\IT \acts M$ with $M^\IT=N$.
Let $\alpha \colon \IZ \acts M$ be any action commuting with the gauge action.
(For instance, any quasi-free action on $\mathcal{O}_2$ extends to such an action.)
Fix $\chi \in \widehat{\IT} \setminus\{1\}$.
We set
\[C:=\{c\in M \vnc{\alpha} \IZ: \E_n(c) \in M_{\chi^n} {\rm ~for~}n\in \IZ\}\in \Int{N}{M\vnc{\alpha}\IZ}.\]
Then one can show that $\E_n(C)\cap M^{\rm u}=\emptyset$ for $n\neq 0$
and hence $C$ does not admit a natural crossed product splitting.
The proof is the same as the \Cs-algebra case discussed in Section \ref{section:O2}, hence we omit it.
\end{Rem}

\section{Galois's type result for compact-by-discrete groups}\label{section:Galois}
Throughout this section, we let $K$ be a compact second countable group, $\Gamma$ be a countable discrete group, and $A$ be a separable \Cs-algebra.
To avoid confusion, in this section, we write $g, h$ for elements in $K$, and write $s, t$ for elements in $\Gamma$.

By applying our splitting theorem, we establish a Galois's type theorem for the Bisch--Haagerup type inclusions \cite{BH} $A^K\subset A \rca{\alpha} \Gamma$
arising from semidirect product group actions $\alpha\colon K\rtimes_\sigma \Gamma \acts A$. 
The precise statement is as follows. We first introduce some notations.

We denote by $ {\rm D}(K)$ the closure of the commutator subgroup of $K$.
We let $K^{\rm ab}$ denote the Hausdorff abelianization of $K$,
namely
\[K^{\rm ab}:=K/ {\rm D}(K).\]
To avoid heavy notations, for $\alpha \colon K \acts A$,
we write the induced action $K^{\rm ab} \acts A^{{\rm D}(K)}$ by the same symbol $\alpha$.
Clearly one has $(A^{\rm D(K)})^{K^{\rm ab}}=A^K$.

Consider an action $\alpha \colon K \acts A$.
For $L <K$ and $\chi \in \widehat{L^{\rm ab}}$, we define
\[A_\chi:=\{a\in A: \alpha_g(a)=\chi_g \cdot a \quad{\rm~for~}g\in L\}.\]
Here and below we often identify $\chi\in \widehat{L^{\rm ab}}$
with a homomorphism $L \rightarrow \IT$ via the quotient map.
Note that $A_\chi \subset A^{\ker(\chi)} \subset A^{{\rm D}(L)}$.
Although the induced action $K \acts \cM(A)$ is not necessary continuous in norm,
we also consider the subspaces $\cM(A)_\chi \subset \cM(A)$ for $\chi\in  \widehat{L^{\rm ab}}$.
By the same way as the abelian case (Section \ref{subsection:compact}), for any $\chi \in \widehat{L^{\rm ab}}$,
we have a canonical projection $\E^L_{\chi} \colon A \rightarrow A_\chi$ given by
\[\E^L_\chi(a):=\int_L \alpha_g(a)\overline{\chi_{g}} {\rm d}m_L(g).\]

Associated to each action $\sigma \colon \Gamma \rightarrow {\rm Aut}(K)$, we define the (complete) lattice $\cL(K, \Gamma, \sigma)$ as follows.
First, as a set, it is defined to be
\[ \cL(K, \Gamma, \sigma):=\left\{ (L, \Lambda, \omega): \begin{array}{l}
 L<K,\quad \Lambda < \Gamma,\\ L {\rm~is~}\sigma(\Lambda){\rm \text-invariant}, \\
 \omega\in {{\rm Z}^1}(\Lambda, \widehat{L^{\rm ab}}; \sigma^\ast)
 \end{array}
 \right\}.\]
 Here ${\rm Z}^1(\Lambda, \widehat{L^{\rm ab}}; \sigma^\ast)$ denotes the group of all $1$-cocycles of the induced action
 $\sigma^\ast \colon \Lambda \acts \widehat{L^{\rm ab}}$; $\sigma^\ast_s(\chi):=\chi\circ\sigma_s^{-1}$.
 Namely,
\[{\rm Z}^1(\Lambda, \widehat{L^{\rm ab}}; \sigma^\ast):=\left\{\omega \colon \Lambda \rightarrow \widehat{L^{\rm ab}}: \omega(s)\sigma^\ast_s(\omega(t))=\omega(st) \quad {\rm~for~}s, t\in \Lambda\right\}.\]
 On $\cL(K, \Gamma, \sigma)$, we define the partial order $\preceq$ by declaring
$(L_1, \Lambda_1, \omega_1) \preceq (L_2, \Lambda_2, \omega_2)$
when $L_1 > L_2$, $\Lambda_1 < \Lambda_2$, and
the diagram
\[
 \begin{CD}
 \Lambda_1 @>{\omega_1}>> \widehat{L_1^{\rm ab}}\\
 @VVV @VVV \\
 \Lambda_2 @>{\omega_2}>> \widehat{L_2^{\rm ab}}
 \end{CD}\]
 commutes.
Here the left and right vertical maps are the inclusion map and the pullback of
the induced homomorphism $L_2^{\rm ab} \rightarrow L_1^{\rm ab}$ respectively.
It is easy to check that $(\cL(K, \Gamma, \sigma), \preceq)$ is indeed a complete lattice.

We now state our Galois's type theorem.
For $(L, \Lambda, \omega)\in \cL(K, \Gamma, \sigma)$,
we define
\[\Cso(L, \Lambda, \omega):=\overline{\rm span}\{A_{\omega(s)}s: s\in \Lambda\} \in \Int{A^K}{A\rca{\alpha}\Gamma}.\]
By the $1$-cocycle equation of $\omega$, the linear span forms a $\ast$-subalgebra.
Hence indeed we have $\Cso(L, \Lambda, \omega) \in \Int{A^K}{A\rca{\alpha}\Gamma}$.

By combining the Galois correspondence theorem for isometrically shift-absorbing actions \cite{GS2} of compact groups
recently shown by Mukohara \cite{Mu} (see also \cite{Izu} for the finite group case) and our crossed product splitting theorem, we now establish the following Galois's type correspondence result.
\begin{Thm}[Theorem \ref{Thmint:Galois}]\label{Thm:Galois}
Let $\alpha \colon K \rtimes_\sigma \Gamma \acts A$ be an action on a simple separable \Cs-algebra.
Assume that $\id_{\mathcal{O}_2} \otimes \alpha$ is isometrically shift-absorbing.
Then the map
\[(L, \Lambda, \omega) \mapsto \Cso(L, \Lambda, \omega)\]
gives a lattice isomorphism
\[\cL(K, \Gamma, \sigma) \cong \Int{A^K}{A\rca{\alpha} \Gamma}.\]
\end{Thm}

We divide the proof into a few lemmas.
The inverse map is obtained in the proof.

For a \Cs-algebra $A$,
let $A_\infty:=\ell^\infty(\IN, A)/c_0(\IN, A)$.
For a continuous group action $\alpha \colon G \acts A$,
let $\alpha_\infty \colon G \acts A_\infty$ denote the induced (possibly discontinuous) action.
Denote by $A_{\infty, \alpha}$ the
set of all $G$-continuous elements with respect to $\alpha_\infty$.
\begin{Lem}\label{Lem:isa}
Let $\alpha \colon K \rtimes_\sigma \Gamma \acts A$ be an isometrically shift-absorbing action on a separable \Cs-algebra $A$.
Then $A^K \subset (A, \alpha|_\Gamma)$ is centrally $\Gamma$-free.
\end{Lem}
\begin{proof}
By a characterization of the isometrically shift-absorption (\cite{GS2}, Proposition 3.8 (iv)), we have a ($K\rtimes_\sigma \Gamma$)-equivariant $A$-bilinear map
\[\theta \colon (L^2(K\rtimes_\sigma \Gamma, A), \widetilde{\alpha}) \rightarrow A_{\infty, \alpha}\]
with
\[\theta (\xi)^\ast \theta (\eta)=\ip{\xi, \eta}_A.\]
Since $\theta $ is ($K\rtimes_\sigma \Gamma$)-equivariant, it restricts to the $\Gamma$-equivariant $A^K$-bilinear map
\[\ell^2(\Gamma, A^\Gamma) \rightarrow (A_{\infty, \alpha})^K.\]
Here we identify $\ell^2(\Gamma)$ with the $K$-fixed point space $L^2(K\rtimes_\sigma \Gamma)^K$ in the obvious way.
Observe that $(A_{\infty, \alpha})^K=(A^K)_\infty$.
Indeed any representing sequence of an element in $A_{\infty, \alpha}$
is equicontinuous with respect to the $K$-action (by \cite{Bro}).
Thus, by applying $\E^K$ termwise, we obtain
$(A_{\infty, \alpha})^K=(A^K)_\infty$.

We pick an approximate unit $(e_n)_{n\in \IN}$ of $A^K \subset A$.
By the choice of $\theta$, for $m\in \IN$, the element $v_m:=\theta(\chi_{K\times\{e\}}\otimes e_m )\in (A^K)_\infty$ satisfies
\[v_m^\ast a v_m= e_m a e_m, \quad v_m^\ast a \alpha_s(v_m)=0\quad{\rm~for~all~}a\in A, s\in \Gamma \setminus\{e\}.\]
For each $m\in \IN$, we choose a representing sequence $(x^{(m)}_n)_{n\in \IN}\subset (A^K)_1$ of $v_m$.
Then, by the above conditions,
for any $a, b\in A$, any $s\in \Gamma \setminus \{e\}$, and any $\epsilon>0$,
one has
\[(x^{(m)}_n)^\ast a x^{(m)}_n \approx_\epsilon a, \quad (x^{(m)}_n)^\ast b \alpha_s(x^{(m)}_n)\approx_{\epsilon} 0\]
for suitable choices of $n, m\in \IN$.
This confirms the central $\Gamma$-freeness of $A^K\subset (A, \alpha|_\Gamma)$.
\end{proof}
The next lemma is necessary for the injectivity of our map
\[\cL(K, \Gamma, \sigma) \rightarrow \Int{A^K}{A\rca{\alpha} \Gamma}.\]
\begin{Lem}\label{Lem:ev}
Let $K \acts A$ be an action on a \Cs-algebra such that the inclusion $A^K\subset A$
is \Cs-irreducible and satisfies the Galois correspondence.
Then for any $L<K$ and any $\chi\in \widehat{L^{\rm ab}}$, one has $A_\chi \neq \0$.\end{Lem}
\begin{proof}
Note that the action $K \acts A$ induces an action $L^{\rm ab} \acts A^{ {\rm D}(L)}$,
which still satisfies the assumptions.
Hence it suffices to show the statement when $K=L^{\rm ab}$.
By taking the tensor product with $\mathcal{O}_2$ if necessary, we may assume that $\alpha$ is of the form $\id_{\mathcal{O}_2}\otimes\alpha_0$.
Then, for each $\tau \in \widehat{K}$,
one has either $A_\tau=\{0\}$ or the strict closure of $A_\tau$ contains an element in $\cM(A)^{\rm u}$
by a similar argument to the proof of Theorem \ref{Thm:generalCs}.
This shows that the set
\[\Upsilon:=\{\tau \in \widehat{K}: A_\tau\neq \{0\}\}\]
forms a subgroup of $\widehat{K}$.
It suffices to show that $\Upsilon=\widehat{K}$.
Set $H:=\bigcap_{\tau \in \Upsilon} \ker(\tau).$
Since $A=\overline{{\rm span}}\{A_\tau: \tau \in \Upsilon\}$ (see e.g., Section 8.1 in \cite{Pedbook}),
one has
$A^{H} = A$.
By (the injectivity of) the Galois correspondence, one has $H=\{1\}$.
Hence $\Upsilon=\widehat{K}$.
\end{proof}

As discussed below Lemma \ref{Lem:AP}, for \Cs-algebras, when the acting group $\Gamma$ does not have the AP,
in general it is difficult to determine which elements in the crossed product sit in a given twisted crossed product \Cs-subalgebra.
Fortunately in the present situation, thanks to the next lemma,
this problem does not appear.
(When $\Gamma$ has the AP, one can skip this lemma by Lemma \ref{Lem:AP}.)

\begin{Lem}\label{Lem:exp}
Let $\alpha \colon K \rtimes_\sigma \Gamma \acts A$ be an action on a \Cs-algebra.
Let $\mathfrak{l}=(L, \Lambda, \omega) \in \cL(K, \Gamma, \sigma)$.
Then there is a conditional expectation
\[\E_\mathfrak{l} \colon A \rca{\alpha} \Gamma \rightarrow \Cso(L, \Lambda, \omega)\]
satisfying 
\[\E_\mathfrak{l}(as) = \left\{
\begin{array}{ll} \E^{L}_{\omega(s)}(a)s &{\rm ~when~} s\in \Lambda\\
0 & {\rm~otherwise}.
\end{array}
\right.\]
Thus, $x\in A \rca{\alpha} \Gamma$ sits in $\Cso(L, \Lambda, \omega)$ if and only if it satisfies
\[\E_s(x)\in A_{\omega(s)}\quad {\rm~for~}s\in \Lambda \quad{\rm~and~}\quad \E_s(x)=0 \quad{\rm~for~}s\in \Gamma \setminus \Lambda.\]
\end{Lem}
\begin{proof}
We have the subgroup conditional expectation
\[\E^\Gamma_\Lambda \colon A \rca{\alpha} \Gamma \rightarrow A \rca{\alpha}\Lambda\]
given by $\E^\Gamma_\Lambda(as)=\chi_\Lambda(s)as$
and the conditional expectation
\[\widetilde{\E}^{{\rm D}(L)} \colon A \rca{\alpha}\Lambda \rightarrow A^{ {\rm D}(L)} \rca{\alpha}\Lambda\]
given by $\widetilde{\E}^{ {\rm D}(L)}(as):=\E^{{\rm D}(L)}(a)s$
(which is well-defined since
$\E^{ {\rm D}(L)} \colon A \rightarrow A^{ {\rm D}(L)}$ is $\Lambda$-equivariant).
Hence we only need to construct the desired conditional expectation on $A^{{\rm D}(L)}\rca{\alpha} \Lambda$.
(Indeed the composition with $\widetilde{\E}^{{\rm D}(L)} \circ \E^\Gamma_\Lambda$ then gives the desired map.)

To give the desired map, we take a faithful covariant representation $(\pi, \fv)$ of the inherited action $L^{\rm ab} \acts A^{{\rm D}(L)}$ on a Hilbert space $\fH$.
Let $\Pi \colon A^{{\rm D}(L)}\rca{\alpha} \Lambda \rightarrow \IB(\fH\otimes \ell^2(\Lambda))$ be the regular representation defined via $\pi$:
\[\Pi(a)(\xi \otimes \delta_t):=\pi(\alpha_{t}^{-1}(a))\xi \otimes \delta_t, \quad \Pi(s)(\xi \otimes \delta_t):=\xi\otimes \delta_{st}\]
for $a\in A$, $s, t\in \Lambda$, $\xi \in \fH$.
For each $g\in L^{\rm ab}$, we define a unitary operator $\mathfrak{U}_g$ on $\fH\otimes \ell^2(\Lambda)$
by
\[\fU_g(\xi \otimes \delta_s):=\overline{\omega(s)_{\sigma_s(g)}} \cdot \fv_g(\xi)\otimes \delta_s.\]
Note that the map $g\mapsto \fU_g$ is strongly continuous.
For $a\in A$, $s, t\in \Lambda$, $\xi \in \fH$, and $g\in L^{\rm ab}$, direct calculations show that
\begin{align*}
\fU_g\Pi(as)\fU_g^\ast (\xi \otimes \delta_t)&=\left[\overline{\omega(s)_{\sigma_{st}(g)}} \cdot \pi(\alpha_{st}^{-1}(\alpha_{\sigma_{st}(g)}(a)))\xi \right]\otimes \delta_{st}\\
&=\Pi(\overline{\omega(s)_{\sigma_{st}(g)}}\cdot \alpha_{\sigma_{st}(g)}(a)s)(\xi \otimes \delta_t).
\end{align*}
Here we use the equation
\[\overline{\omega(st)_{\sigma_{st}(g)}} \cdot \omega(t)_{\sigma_{t}(g)}=\overline{\omega(s)_{\sigma_{st}(g)}},\]
which follows from the $1$-cocycle equation of $\omega$.
Observe that
\[\int_{L^{\rm ab}} \overline{\omega(s)_{\sigma_{st}(g)}}\cdot \alpha_{\sigma_{st}(g)}(a) {\rm d}m_{L^{\rm ab}}(g) = \E^{L^{\rm ab}}_{\omega(s)}(a)\]
for all $s, t\in \Lambda$ and $a\in A^{{\rm D}(L)}$,
where the Bochner integral is taken in the norm topology.
Thus, after identifying $A^{{\rm D}(L)}\rca{\alpha} \Lambda$
with a \Cs-subalgebra of $\IB(\fH\otimes \ell^2(\Lambda))$ via $\Pi$,
the desired map is obtained
as the co-restriction of the Bochner integral
\[\int_{L^{\rm ab}}\ad(\fU_g) {\rm d}m_{L^{\rm ab}}(g)\]
in $\IB(A^{{\rm D}(L)}\rca{\alpha} \Lambda ,\IB(\fH\otimes \ell^2(\Lambda)))$
with respect to the point strong operator topology.

For the last statement, if $x\in A \rca{\alpha} \Gamma$ satisfies the conditions,
then one has $\E_s(\E_\mathfrak{l}(x))=\E_s(x)$ for all $s\in \Gamma$,
whence $x=\E_\mathfrak{l}(x)\in \Cso(L, \Lambda, \omega)$. The converse is trivial.
\end{proof}

The next lemma excludes the existence of intermediate \Cs-algebras not covered by $\cL(K, \Gamma, \sigma)$.
\begin{Lem}\label{Lem:conj}
Let $\alpha \colon K \acts A$ be an action.
Assume that $A^K \subset A$ is irreducible and satisfies the Galois correspondence.
Assume that $u\in \cM(A)^{\rm u}$ and $L, M<K$ satisfy
\[uA^L u^\ast = A^M.\]
Then $L=M$ and $u\in \cM(A)_\chi$ for $\chi\in \widehat{L^{\rm ab}}$ given by $\chi_g:=u^\ast \alpha_g(u)$ for $g\in L$.
\end{Lem}
\begin{proof}
The assumption implies
$uA^K u^\ast \subset A^M$.
Hence for any $x\in A^K$ and any $g\in M$, one has
\[u x u^\ast=\alpha_g(u x u^\ast )=\alpha_g(u)x\alpha_g(u)^\ast.\]
This shows that $\alpha_g(u)^\ast u\in (A^K)'\cap \cM(A)=\IC$.
Hence $u\in \cM(A)_\chi$ for $\chi\in \widehat{M^{\rm ab}}$ defined by $\chi_g:=u^\ast \alpha_g(u)$ for $g\in M$.
This proves $u^\ast A^M u = A^M$.
Since $u^\ast A^M u =A^L$, one has $A^L=A^M$ and hence the Galois correspondence yields $L=M$.
\end{proof}

Now we are able to complete the proof.
\begin{proof}[Proof of Theorem \ref{Thm:Galois}]
Thanks to Lemmas \ref{Lem:isa} and \ref{Lem:ev} and the Galois correspondence for $A^K \subset A$,
clearly the map in the statement is injective.
The relation $(L_1, \Lambda_1, \omega_1)\preceq (L_2, \Lambda_2, \omega_2)$ for two elements in $\cL(K, \Gamma, \sigma)$
implies $\Lambda_1<\Lambda_2$ and $A_{\omega_1(s)}\subset A_{\omega_2(s)}$ for all $s\in \Lambda_1$.
Hence $\Cso(L_1, \Lambda_1, \omega_1)\subset \Cso(L_2, \Lambda_2, \omega_2)$.
Conversely, if $\Cso(L_1, \Lambda_1, \omega_1)\subset \Cso(L_2, \Lambda_2, \omega_2)$,
one clearly has $A^{L_1} \subset A^{L_2}$ and $\Lambda_1 \subset \Lambda_2$ (by Lemma \ref{Lem:ev}).
By the Galois correspondence \cite{Mu}, the first relation implies $L_2<L_1$.
Also, for any $s\in \Lambda_1$, one has
\[A_{\omega_1(s)}=\E_s(\Cso(L_1, \Lambda_1, \omega_1))\subset \E_s(\Cso(L_2, \Lambda_2, \omega_2))
=A_{\omega_2(s)}.\]
This together with Lemma \ref{Lem:ev} implies $\omega_1(s)|_{L_2}=\omega_2(s)$ for $s\in \Lambda_1$.
This proves $(L_1, \Lambda_1, \omega_1)\preceq (L_2, \Lambda_2, \omega_2)$,
and hence the map is an order isomorphism onto the image.

For the surjectivity, it suffices to show the statement for $\id_{\mathcal{O}_2}\otimes \alpha$ instead of $\alpha$.
Hence we may assume that $\alpha$ is of the form $\id_{\mathcal{O}_2}\otimes\alpha_0$.

Let $C\in \Int{A^K}{A\rca{\alpha}\Gamma}$.
We need to show that $C=\Cso(L, \Lambda, \omega)$ for some $(L, \Lambda, \omega) \in \cL(K, \Gamma, \sigma)$.
We will explicitly give the desired triplet.

By the Galois correspondence for $A^K \subset A$ \cite{Mu}, the \Cs-algebra $A \cap C$ is of the form $A^L$ for a unique $L <K$.
Put $\Lambda:=\{s\in \Gamma: \E_s(C)\neq \0\}$.
By Theorem \ref{Thm:generalCs} and \ref{Lem:isa} (recall that $\alpha$ is of the form $\id_{\mathcal{O}_2}\otimes\alpha_0$),
for each $s\in \Lambda$, one can choose $\fu_s\in \cM(A)^{\rm u}$ with $\E_s(C)s=A^L \fu_s s= C\cap As$.
Observe that
\[A^L=\fu_ s s A^L (\fu_s s)^\ast=\fu_s A^{\sigma_s(L)} \fu_s^\ast.\]
By Lemma \ref{Lem:conj}, this implies $\sigma_s(L)=L$ for $s\in \Lambda$.
Moreover by Lemma \ref{Lem:conj},
one has a map $\omega \colon \Lambda \rightarrow \widehat{L^{\rm ab}}$
satisfying $\fu_s\in \cM(A)_{\omega(s)}$ for $s\in \Lambda$.
Since $A_{\omega(s)}=A^L \fu_s$, we conclude $\E_s(C)=A_{\omega(s)}$ for $s\in \Lambda$.
As
\[A_{\omega(s)} \alpha_s(A_{\omega(t)})st=A_{\omega(s)} s A_{\omega(t)} t \subset C \quad {\rm~ and~}\quad A_{\omega(s)} \alpha_s(A_{\omega(t)}) \subset A_{\omega(s)\sigma^\ast_s(\omega(t))}\]
for $s, t\in \Lambda$,
one has $\omega \in {\rm Z}^1(\Lambda, \widehat{L^{\rm ab}}; \sigma^\ast)$ by Lemma \ref{Lem:ev}.

Now it is clear that $(L, \Lambda, \omega) \in \cL(K, \Gamma, \sigma)$ and
that $\Cso(L, \Lambda, \omega) \subset C$.
The reverse inclusion follows from Lemma \ref{Lem:exp}.
\end{proof}

\begin{Rem}
In the case both $K$ and $\Gamma$ are finite,
the statement follows from the finite quantum group Galois corresponding theorem of Izumi \cite{Izu}.
To see this, apply the theorem to the obvious quantum group action of $K \rtimes \widehat{\Gamma}$ on $A \rc \Gamma$. 
\end{Rem}

\begin{Rem}
By a similar way, one can prove the analogous Galois's type result for von Neumann algebras.
However this also follows from the Galois correspondence for Kac algebras shown by Izumi--Longo--Popa \cite{ILP}.
To see this, apply their theorem to the obvious action of $K\rtimes_{\sigma}\widehat{\Gamma}$ on $M\vnc{\alpha}\Gamma$.
We therefore do not give details.
\end{Rem}
\subsection*{Acknowledgements}
The author is grateful to the organizers of \emph{The First International Congress of Basic Science} at Beijing (2023)
for warmful invitation to the conference and awarding \emph{Frontiers of Science Awards} to my paper \cite{SuzCMP}.
These stimulating experiences motivated me to revisit phenomena around this work.

The present work was supported by JSPS KAKENHI (Grant-in-Aid for Early-Career Scientists)
Grant Numbers JP19K14550, JP22K13924.

\subsection*{Conflict of interest}
 On behalf of all authors, the corresponding author states that there is no conflict of interest.

\end{document}